\documentclass[11pt,a4paper]{amsart}
\usepackage{vmargin}
\setmargrb{1.8cm}{2.8cm}{1.8cm}{1.8cm}

\usepackage{amsfonts, amsmath, amssymb, amsthm}
\usepackage[mathscr]{eucal} 
\usepackage{hyperref}

\usepackage{graphicx,color}
\usepackage{booktabs}
\usepackage{bbm}
\usepackage{overpic}
\usepackage{xspace}

\providecommand{\abs}[1]{\left\lvert#1\right\rvert}

\providecommand{\NN}{{\mathbb{N}}}

\providecommand{\ZZ}{{\mathbb{Z}}}
\providecommand{\RR}{{\mathbb{R}}}
\providecommand{\nonnegRR}{{\RR_{\ge0}}}
\providecommand{\CC}{{\mathbb{C}}}
\providecommand{\Sph}{{\mathbb{S}}}
\providecommand{\cA}{{\mathscr{A}}}

\providecommand{\cG}{{\mathscr{G}}}

\providecommand{\cM}{{\mathscr{M}}}

\providecommand{\cS}{{\mathscr{S}}}
\providecommand{\cT}{{\mathscr{T}}}

\providecommand{\vones}{\mathbbm{1}}

\providecommand{\SetOf}[2]{\left\{#1\vphantom{#2}\,\right.\left|\,\vphantom{#1}#2\right\}}
\providecommand{\smallSetOf}[2]{\{#1\,|\,#2\}}

\DeclareMathOperator{\lin}{lin}
\DeclareMathOperator{\aff}{aff}
\DeclareMathOperator{\conv}{conv}
\DeclareMathOperator{\pos}{pos}

\DeclareMathOperator{\image}{im}
\providecommand{\smallpmatrix}[1]{\left(\begin{smallmatrix} #1 \end{smallmatrix}\right)}

\providecommand\cprime{$'$}

\providecommand{\envelope}[2]{{{\mathscr{E}}_{#1}(#2)}}
\providecommand{\lift}[2]{{{\mathscr{L}}_{#1}(#2)}}
\providecommand{\tightspan}[2]{{{\mathscr{T}}_{#1}(#2)}}
\providecommand{\subdivision}[2]{{\Sigma_{#1}(#2)}}
\providecommand{\Hypersimplex}[2]{{\Delta(#1,#2)}}
\providecommand{\Grassmannian}[2]{{G_{#1,#2}}}
\providecommand{\tropGrassmannian}[2]{{\cG_{#1,#2}}}
\providecommand{\tropGrassmannianOne}[2]{{\cG'_{#1,#2}}}
\providecommand{\tropGrassmannianTwo}[2]{{\cG''_{#1,#2}}}
\providecommand{\tropGrassmannianThree}[2]{{\cG'''_{#1,#2}}}
\providecommand{\tropPreGrassmannian}[2]{{{\rm pre-}\cG_{#1,#2}}}
\providecommand{\tropPreGrassmannianOne}[2]{{{\rm pre-}\cG'_{#1,#2}}}
\providecommand{\Tropicalization}[1]{\cT(#1)}
\providecommand{\transpose}[1]{{#1}^{\sf T}}
\providecommand{\scp}[2]{\langle{#1},{#2}\rangle}
\providecommand{\PrimeWeights}[1]{{\mathscr W}(#1)}
\providecommand{\SplitComplex}[1]{{\rm Split}(#1)}
\providecommand{\WeakSplitComplex}[1]{{\rm Split}^{\rm w}(#1)}
\DeclareMathOperator{\Vertices}{Vert}

\DeclareMathOperator{\Assoc}{Assoc}
\DeclareMathOperator{\vol}{Vol}
\DeclareMathOperator{\initial}{in}

\providecommand{\SecondaryFan}[1]{{\rm SecFan}(#1)}
\providecommand{\SecondaryFanOne}[1]{{\rm SecFan}'(#1)}
\providecommand{\SecondaryPolytope}[1]{{\rm SecPoly}(#1)} 
\providecommand{\SplitPoly}[1]{{\rm SplitPoly}(#1)}
\providecommand{\Quotient}{\slash}

{\end{list}}

\theoremstyle{plain}
\newtheorem{theorem}{Theorem}
\newtheorem{proposition}[theorem]{Proposition}
\newtheorem{corollary}[theorem]{Corollary}
\newtheorem{lemma}[theorem]{Lemma}

\newtheorem{observation}[theorem]{Observation}

\theoremstyle{definition}

\newtheorem{example}[theorem]{Example}
\newtheorem{construction}[theorem]{Construction}

\newtheorem{remark}[theorem]{Remark}
\newtheorem{question}[theorem]{Question}

\graphicspath{ {Pictures/} }

\begin{document}

\title{Splitting Polytopes}

\author[Herrmann and Joswig]{Sven Herrmann \and Michael Joswig}
\address{Fachbereich Mathematik, TU Darmstadt, 64289 Darmstadt, Germany}
\thanks{Sven Herrmann is supported by a Graduate Grant of TU Darmstadt. Research by Michael Joswig is supported by DFG
  Research Unit ``Polyhedral Surfaces''.}
\email{\{sherrmann,joswig\}@mathematik.tu-darmstadt.de}

\date{\today}

\begin{abstract}
  A \emph{split} of a polytope $P$ is a (regular) subdivision with exactly two maximal
  cells.  It turns out that each weight function on the vertices of $P$ admits a unique
  decomposition as a linear combination of weight functions corresponding to the splits of
  $P$ (with a split prime remainder).  This generalizes a result of Bandelt and Dress
  [Adv.\ Math.\ 92 (1992)] on the decomposition of finite metric spaces.

  Introducing the concept of \emph{compatibility} of splits gives rise to a finite
  simplicial complex associated with any polytope $P$, the \emph{split complex} of $P$.
  Complete descriptions of the split complexes of all hypersimplices are obtained.
  Moreover, it is shown that these complexes arise as subcomplexes of the tropical
  (pre-)Grassmannians of Speyer and Sturmfels [Adv. Geom. 4 (2004)].
\end{abstract}

\maketitle

\section{Introduction}

A real-valued weight function $w$ on the vertices of a polytope $P$ in $\RR^d$ defines a
polytopal subdivision of $P$ by way of lifting to $\RR^{d+1}$ and projecting the lower
hull back to $\RR^d$.  The set of all weight functions on $P$ has the natural structure of
a polyhedral fan, the \emph{secondary fan} $\SecondaryFan{P}$.  The rays of
$\SecondaryFan{P}$ correspond to the coarsest (regular) subdivisions of $P$.  This paper
deals with the coarsest subdivisions with precisely two maximal cells.  These are called
\emph{splits}.

Hirai proved in \cite{Hirai06} that an arbitrary weight function on $P$ admits a canonical
decomposition as a linear combination of split weights with a \emph{split prime}
remainder.  This generalizes a classical result of Bandelt and Dress \cite{BandeltDress92}
on the decomposition of finite metric spaces, which proved to be useful for applications
in phylogenomics; e.g., see Huson and Bryant \cite{HusonBryant06}.  We give a new proof of
Hirai's split decomposition theorem which establishes the connection to the theory of
secondary fans developed by Gel{\cprime}fand, Kapranov, and Zelevinsky \cite{GKZ94}.

Our main contribution is the introduction and the study of the \emph{split complex} of a
polytope $P$.  This comes about as the clique complex of the graph defined by a
\emph{compatibility} relation on the set of splits of~$P$.  A first example is the
boundary complex of the polar dual of the $(n-3)$-dimensional associahedron, which is
isomorphic to the split complex of an $n$-gon.  A focus of our investigation is on the
hypersimplices $\Hypersimplex{k}{n}$, which are the convex hulls of the $0/1$-vectors of
length $n$ with exactly $k$ ones.  We classify all splits of the hypersimplices together
with their compatibility relation.  This describes the split complexes of the
hypersimplices.

Tropical geometry is concerned with the \emph{tropicalization} of algebraic varieties.  An
important class of examples is formed by the \emph{tropical Grassmannians}
$\tropGrassmannian{k}{n}$ of Speyer and Sturmfels \cite{SpeyerSturmfels04}, which are the
tropicalizations of the ordinary Grassmannians of $k$-dimensional subspaces of an
$n$-dimensional vector space (over some field).  It is a challenge to obtain a complete
description of $\tropGrassmannian{k}{n}$ even for most fixed values of $k$ and $n$.  A
better behaved close relative of $\tropGrassmannian{k}{n}$ is the \emph{tropical
  pre-Grassmannian} $\tropPreGrassmannian{k}{n}$ arising from tropicalizing the ideal of
quadratic Pl\"ucker relations.  This is a subfan of the secondary fan of
$\Hypersimplex{k}{n}$, and its rays correspond to coarsest subdivisions of
$\Hypersimplex{k}{n}$ whose (maximal) cells are matroid polytopes; see Kapranov
\cite{Kapranov93} and Speyer \cite{Speyer04}.  As one of our main results we prove that
the split complex of $\Hypersimplex{k}{n}$ is a subcomplex of
$\tropPreGrassmannianOne{k}{n}$, the intersection of the fan $\tropPreGrassmannian{k}{n}$
with the unit sphere in $\RR^{\tbinom{n}{k}}$.  Moreover, we believe that our approach can
be extended further to obtain a deeper understanding of the tropical (pre-)Grassmannians.
To follow this line, however, is beyond the scope of this paper.

The paper is organized as follows. We start out with the investigation of general weight
functions on a polytope $P$ and their coherence.  Two weight functions are \emph{coherent}
if there is a common refinement of the subdivisions that they induce on $P$.  As an
essential technical device for the subsequent sections we introduce the \emph{coherency
  index} of two weight functions on $P$.  This generalizes the definition of Koolen and
Moulton for $\Hypersimplex{2}{n}$ \cite{KoolenMoulton96}, Section 4.1.

The third section then deals with splits of polytopes and the corresponding weight
functions.  As a first result we give a concise new proof of the split decomposition
theorems of Bandelt and Dress \cite{BandeltDress92}, Theorem~3, and Hirai \cite{Hirai06},
Theorem~2.2.

A split subdivision of the polytope $P$ is clearly determined by the affine hyperplane
spanned by the unique interior cell of codimension $1$.  A set of splits is
\emph{compatible} if any two of the corresponding split hyperplanes do not meet in the
(relative) interior of $P$.  The \emph{split complex} $\SplitComplex{P}$ is the abstract
simplicial complex of compatible sets of splits of $P$.  It is an interesting fact that
the subdivision of $P$ induced by a sum of weights corresponding to a compatible system of
splits is dual to a tree.  In this sense $\SplitComplex{P}$ can always be seen as a ``space of
trees''.

In Section~\ref{sec:hypersimplices} we study the hypersimplices $\Hypersimplex{k}{n}$.
Their splits are classified and explicitly enumerated.  Moreover, we characterize the
compatible pairs of splits.  The purpose of the short Section~\ref{sec:metric} is to
specialize our results for arbitrary hypersimplices to the case $k=2$.  A metric on a
finite set of $n$ points yields a weight function on $\Hypersimplex{2}{n}$, and hence all
the previous results can be interpreted for finite metric spaces.  This is the classical
situation studied by Bandelt and Dress \cite{BandeltDress86,BandeltDress92}.  Notice that
some of their results had already been obtained by Isbell much earlier \cite{Isbell64}.

Section~\ref{sec:matroid} bridges the gap between the split theory of the hypersimplices
and matroid theory.  This way, as one key result, we can prove that the split complex of
the hypersimplex $\Hypersimplex{k}{n}$ is a subcomplex of the tropical pre-Grassmannian
$\tropPreGrassmannianOne{k}{n}$.  We conclude the paper with a list of open questions.

\section{Coherency of Weight Functions}
\label{sec:coherency}

Let $P\subset\RR^{d+1}$ be a polytope with vertices $v_1,\dots,v_n$. We form the $n\times
(d+1)$-matrix $V$ whose rows are the vertices of $P$.  For technical reasons we make the
assumption that $P$ is $d$-dimensional and that the (column) vector $\vones:=(1,\dots,1)$
is contained in the linear span of the columns of~$V$. In particular, this implies that
$P$ is contained in some affine hyperplane which does not contain the origin.  A
\emph{weight function} $w : \Vertices{P} \to \RR$ of $P$ can be written as a vector in
$\RR^n$. Now each weight function $w$ of $P$ gives rise to the unbounded polyhedron
\[
\envelope{w}{P} \ := \ \SetOf{x\in\RR^{d+1}}{Vx \ge -w} \, ,
\]
the \emph{envelope} of $P$ with respect to~$w$.  We refer to Ziegler \cite{Ziegler95} for
details on polytopes.

If $w_1$ and $w_2$ are both weight functions of $P$, then $Vx \ge -w_1$ and $Vy \ge -w_2$
implies $V(x+y)\ge-(w_1+w_2)$.  This yields the inclusion
\begin{equation}\label{eq:coherent}
  \envelope{w_1}{P}+\envelope{w_2}{P} \ \subseteq \ \envelope{w_1+w_2}{P} \, .
\end{equation}
If equality holds in \eqref{eq:coherent} then $(w_1,w_2)$ is called a \emph{coherent
  decomposition} of $w=w_1+w_2$. (Note that this must not be confused with the notion of
``coherent subdivision'' which is sometimes used instead of ``regular subdivision''.)

\begin{example}\label{ex:hexagon:1}
  We consider a hexagon $H\subset\RR^3$ whose vertices are the columns of the matrix
  \[
  \transpose{V} \ = \
  \begin{pmatrix}
    1 & 1 & 1 & 1 & 1 & 1 \\
    0 & 1 & 2 & 2 & 1 & 0 \\
    0 & 0 & 1 & 2 & 2 & 1
  \end{pmatrix}
  \]
  and three weight functions $w_1=(0,0,1,1,0,0)$, $w_2=(0,0,0,1,1,0)$, and
  $w_3=(0,0,2,3,2,0)$.  Again we identify a matrix with the set of its rows.  A direct
  computation then yields that $w_1+w_2$ is not coherent, but both $w_1+w_3$ and $w_2+w_3$
  are coherent.
\end{example}

Each \emph{face} of a polyhedron, that is, the intersection with a supporting hyperplane,
is again a polyhedron, and it can be bounded or not.  A polyhedron is \emph{pointed} if it
does not contain an affine subspace or, equivalently, its lineality space is trivial. This implies
that the set of all bounded faces is non-empty and forms a polytopal complex.  This
polytopal complex is always contractible (see Hirai \cite[Lemma 4.5]{Hirai06b}).  The polytopal
complex of bounded faces of the polyhedron $\envelope{w}{P}$ is called the \emph{tight
  span} of $P$ with respect to $w$, and it is denoted by $\tightspan{w}{P}$.

\begin{lemma}\label{lem:coherent}
  Let $w=w_1+w_2$ be a decomposition of weight functions of $P$.  Then the following
  statements are equivalent.
  \begin{enumerate}
  \item\label{it:coherent:def} The decomposition $(w_1,w_2)$ is coherent,
  \item\label{it:coherent:ts-ts} $\tightspan{w}{P}\subseteq\tightspan{w_1}{P}+\tightspan{w_2}{P}$\,,
  \item\label{it:coherent:ts-lift} $\tightspan{w}{P}\subseteq\envelope{w_1}{P}+\envelope{w_2}{P}$\,,
  \item\label{it:coherent:vertex} each vertex of $\tightspan{w}{P}$ can be written as a
    sum of a vertex of $\tightspan{w_1}{P}$ and a vertex of $\tightspan{w_2}{P}$.
  \end{enumerate}
\end{lemma}

For a similar statement in the special case where $P$ is a second hypersimplex (see
Section~\ref{sec:hypersimplices} below) see Koolen and Moulton \cite{KoolenMoulton97},
Lemma 1.2.

\begin{proof}
  If $(w_1,w_2)$ is coherent then by definition $\envelope{w}{P}=\envelope{w_1}{P}+\envelope{w_2}{P}$.
  Each face $F$ of the Minkowski sum of two polyhedra is the Minkowski sum of two faces
  $F_1,F_2$, one from each summand.  Now $F$ is bounded if and only if $F_1$ and $F_2$ are
  bounded.  This proves that \eqref{it:coherent:def} implies \eqref{it:coherent:ts-ts}.

  Clearly, \eqref{it:coherent:ts-ts} implies \eqref{it:coherent:ts-lift}.  Moreover,
  \eqref{it:coherent:ts-lift} implies \eqref{it:coherent:vertex} by the same argument on
  Minkowski sums as above.

  To complete the proof we have to show that \eqref{it:coherent:def} follows from
  \eqref{it:coherent:vertex}.  So assume that each vertex of $\tightspan{w}{P}$ can be
  written as a sum of a vertex of $\tightspan{w_1}{P}$ and a vertex of
  $\tightspan{w_2}{P}$, and let $x\in\envelope{w}{P}$.  Then $x$ can be written as $x=y+r$
  where $y\in\tightspan{w}{P}$ and r is a ray of $\envelope{w}{P}$, that is, 
  $z+\lambda r\in\envelope{w}{P}$ for all $z\in\envelope{w}{P}$ and all $\lambda \geq 0$.
  It follows that $V r\le 0$.  By assumption
  there are vertices $y_1$ and $y_2$ of $\tightspan{w_1}{P}$ and $\tightspan{w_2}{P}$ such
  that $y=y_1+y_2$.  Setting $x_1:=y_1+r$ and $x_2:=y_2$ we have $x=x_1+x_2$ with
  $x_2\in\envelope{w_2}{P}$.  Computing
  \[
  V x_1 \ = \ V (y_1 + r) \ \le \ V y_1 + V r \ \le \ -w_1 + 0 \ = \ -w_1 \,,
  \]
  we infer that $x_1\in\envelope{w_1}{P}$, and hence $w_1$ and $w_2$ are
  coherent.
\end{proof}

We recall basic facts about cone polarity.  For an arbitrary pointed polyhedron
$X\subset\RR^{d+1}$ there exists a unique polyhedral cone $C(X)\subset\RR^{d+2}$ such that
$X=\SetOf{x\in \RR^{d+1}}{(1,x)\in C(P)}$. If $X$ is given in inequality description
$X=\SetOf{x\in\RR^{d+1}}{A x\geq b}$ one has
\[
C(X) \ = \ \SetOf{y\in\RR^{d+2}}{
  \begin{pmatrix}
    1  & 0 \\
    -b & A
  \end{pmatrix} y \ge 0
} \, .
\]
If $X$ is given in a vertex-ray description $P=\conv V+\pos R$ one has
\[
C(X) \ = \ \pos
\begin{pmatrix}
  \vones & V\\
  0 & R
\end{pmatrix} \, .
\]
For any set $M\subseteq\RR^{d+2}$ its cone polar is defined as
$M^\circ:=\smallSetOf{y\in\RR^{d+2}}{\scp{x}{y}\ge 0 \text{ for all $x\in M$}}$. If
$C=\pos A$ is a cone it is easily seen that $C^\circ=\smallSetOf{y\in\RR^{d+2}}{A y\geq
  0}$ and that ${(C^\circ)}^\circ=C$. The cone $C^\circ$ is called the \emph{polar dual}
cone of $C$. Two polyhedra $X$ and $Y$ are \emph{polar duals} if the corresponding cones
$C(X)$ and $C(Y)$ are. The face lattices of dual cones are anti-isomorphic.

For the following our technical assumptions from the beginning come into play.  Again let
$P$ be a $d$-polytope in $\RR^{d+1}$ such that $\vones$ is contained in the column span of
the matrix $V$ whose rows are the vertices of $P$.  The standard basis vectors of
$\RR^{d+1}$ are denoted by $e_1,\dots,e_{d+1}$.

\begin{proposition}\label{prop:duality}
  The polyhedron $\envelope{w}{P}$ is affinely equivalent to the polar dual of the polyhedron
  \begin{equation*}
    \lift{w}{P} \ := \conv\SetOf{v+w(v)e_{d+1}}{v\in\Vertices{P}} \, + \, \nonnegRR e_{d+1}\, .
  \end{equation*}
  Moreover, the face poset of $\tightspan{w}{P}$ is anti-isomorphic to the face poset of
  the interior lower faces (with respect to the last coordinate) of $\lift{w}{P}$.
\end{proposition}

\begin{proof}
  Note first, that by our assumption that $\vones$ is in the column span of $V$, up to a
  linear transformation of $\RR^{d+1}$, we can assume that $V=(\bar V,\vones)$ for an
  $n\times d$-matrix $\bar V$.  This yields
  \[
  C(\envelope{w}{P}) \ = \ \SetOf{x\in\RR^{d+2}}{\begin{pmatrix}
      1 & 0 & 0\\
      w & \bar V& \vones
    \end{pmatrix} x \geq 0}.
    \]
  On the other hand we have
  \[
  C(\lift{w}{P}) \ = \ \pos \begin{pmatrix}
      \vones &\bar V& w \\
      0 & 0 & 1
    \end{pmatrix},
  \]
  which is linearly isomorphic to $\bar C=\pos \smallpmatrix{w & \vones &\bar V\\ 1& 0 &
    0}$ by a coordinate change, so $\envelope{w}{P}$ and $\lift{w}{P}$ are polar duals, up
  to linear transformations.
    
  This way we have obtained an anti-isomorphism of the face lattices of
  $C(\envelope{w}{P})$ and $C(\lift{w}{P})$.  A face~$F$ of $\envelope{w}{P}$ is bounded
  if and only if no generator of $C(\envelope{w}{P})$ with first coordinate equal to zero
  is smaller then $F$ in the face lattice.  In the dual view, this means that the
  corresponding face $F'$ of $\lift{w}{P}$ is greater then a facet which is parallel to
  the last coordinate axis in the face lattice of $C(\lift{w}{P})$. But this exactly means
  that $F'$ is a lower face. So the lattice anti-isomorphism of $C(\envelope{w}{P})$ and
  $C(\lift{w}{P})$ induces a poset anti-isomorphism between $\tightspan{w}{P}$ and the
  interior lower faces of $\lift{w}{P}$.
\end{proof}

The lower faces of $\lift{w}{P}$ (with respect to the last coordinate) are precisely its
bounded faces.  By projecting back to $\aff{P}$ in the $e_{d+1}$-direction, the polytopal
complex of bounded faces of $\lift{w}{P}$ induces a polytopal decomposition
$\subdivision{w}{P}$ of $P$. Note that we only allow the vertices of $P$ as vertices of
any subdivision of $P$. A polytopal subdivision which arises in this way is called
\emph{regular}.  Two weight functions are \emph{equivalent} if they induce the same
subdivision.  This allows for one more characterization extending
Lemma~\ref{lem:coherent}.

\begin{corollary}\label{cor:coherent}
  A decomposition $w=w_1+w_2$ of weight functions of~$P$ is coherent if and only if the
  subdivision $\subdivision{w}{P}$ is the common refinement of the subdivisions
  $\subdivision{w_1}{P}$ and $\subdivision{w_2}{P}$.
\end{corollary}
\begin{proof}
  By Lemma~\ref{lem:coherent}, the decomposition $w_1+w_2$ is coherent if and only if each
  vertex $x$ of $\tightspan{w}{P}$ is the sum of a vertex $x_1$ of $\tightspan{w_1}{P}$
  and a vertex $x_2$ of $\tightspan{w_2}{P}$.  In terms of the duality proved in
  Proposition~\ref{prop:duality} the vertex $x$ corresponds to the maximal cell
  $F_w(x):=\conv\smallSetOf{v\in\Vertices P}{\scp{v}{x}=-w}$ of $\subdivision{w}{P}$.
  Similarly, $x_1$ and $x_2$ corresponds to the cells $F_{w_1}(x_1)$ and $F_{w_2}(x_2)$ of
  $\subdivision{w_1}{P}$ and $\subdivision{w_2}{P}$, respectively. In fact, we have
  $F_w(x)=F_{w_1}(x_1)\cap F_{w_2}(x_2)$, and so $\subdivision{w}{P}$ is the common
  refinement of $\subdivision{w_1}{P}$ and $\subdivision{w_2}{P}$. The converse follows
  similarly.
\end{proof}  

\begin{example}\label{ex:hexagon:2}
  In Example~\ref{ex:hexagon:1} the tight spans of the three weight functions of the
  hexagon are line segments:
  \[
  \tightspan{w_1}{H} \ = \ [0,(1,-1,0)] \, , \quad \tightspan{w_2}{H} \ = \ [0,(1,0,-1)] \, ,
  \quad \text{and} \quad \tightspan{w_3 }{H} \ = \ [0,(1,-1,-1)] \, .
  \]
\end{example}

\begin{remark}
  Interesting special cases of tight spans include the following. Finite metric spaces (on
  $n$ points) give rise to weight functions on the second hypersimplex
  $P=\Hypersimplex{2}{n}$.  In this case the tight span can be interpreted as a ``space'' of
  trees which are candidates to fit the given metric.  This has been studied by Bandelt
  and Dress~\cite{BandeltDress92}, and this is the context in which the name ``tight
  span'' was used first.  See also Section~\ref{sec:metric} below.

  If $P$ is a product of two simplices, the tight span of a lifting function gives rise to
  a \emph{tropical polytope} introduced by Develin and
  Sturmfels~\cite{DevelinSturmfels04}, the cells in the resulting regular decomposition of
  $P$ are the \emph{polytropes} of~\cite{JoswigKulas08}.

  If $P$ spans the affine hyperplane $x_1=1$ and if we consider the weight function defined
  by $w(v)=v_2^2+v_3^2+\dots+v_{d+1}^2$ for each vertex $v$ of $P$ then the tight span
  $\tightspan{w}{P}$ is isomorphic
  to the subcomplex of bounded faces of the Voronoi diagram of $\Vertices{P}$.  All
  maximal cells of the Voronoi diagram are unbounded and hence the tight span is at most
  $(d-1)$-dimensional.  The subdivision $\subdivision{w}{P}$ is then isomorphic to the
  Delone decomposition of $\Vertices{P}$. 
\end{remark}

Let $w$ and $w'$ be weight functions of our polytope $P$.  We want to have a measure which
expresses to what extent the pair of weight functions $(w',w-w')$ deviates from coherence
(if at all).  The \emph{coherency index} of $w$ with respect to $w'$ is defined as
\begin{equation}\label{eq:coherency-index}
  \alpha^w_{w'} \ := \
  \min_{x\in\Vertices{\envelope{w}{P}}}
  \left\{ \max_{x'\in\Vertices{\envelope{w'}{P}}}
    \left\{ \min_{v\in V_{w'}(x')}
      \left\{ 
        \frac{\scp{v}{x}+w(v)}{\scp{v}{x'}+w'(v)}
      \right\}
    \right\}
  \right\}  \, ,
\end{equation}
where $V_{w'}(x')=\SetOf{v\in\Vertices{P}}{\scp{v}{x'}\ne -w'(v)}$. (That is, $V_{w'}(x')$ is the
set of vertices of $P$ that are not contained in the cell dual to $x$.) The name is justified
by the following observation which generalizes Koolen and Moulton \cite[Theorem 4.1]{KoolenMoulton96}.

\begin{proposition}\label{prop:coherent}
  Let $w$ and $w'$ be weight functions of the polytope $P$.  Moreover, let $\lambda\in\RR$
  and $\tilde w:=w-\lambda w'$.  Then $w=\tilde w+\lambda w'$ is coherent if and only if
  $0\le\lambda\le\alpha^w_{w'}$.
\end{proposition}

\begin{proof}
  Assume that $w=\tilde w+\lambda w'$ is coherent.  By Lemma~\ref{lem:coherent} for each
  vertex $x$ of $\envelope{w}{P}$ there is a vertex $x'$ of $\envelope{w'}{P}$ such that
  $x-\lambda x'$ is a vertex of $\envelope{\tilde w}{P}$.  We arrive at the following sequence
  of equivalences:
  \begin{align*}
    x-\lambda x'\in\tightspan{\tilde w}{P} \
    &\Longleftrightarrow \ -w(v)+\lambda w'(v) \ \le \ \scp {v}{x-\lambda x'}
     \quad \text{for all $v\in\Vertices{P}$} \\
    &\Longleftrightarrow \ \lambda(\scp{v}{x'}+w'(v)) \ \le \ \scp{v}{x}+w(v)
     \quad \text{for all $v\in\Vertices{P}$} \\
    &\Longleftrightarrow \ \lambda \ \le \ \frac{\scp{v}{x}+w(v)}{\scp{v}{x'}+w'(v)}
     \quad \text{for all $v\in V_{w'}(x')$} \\
    &\Longleftrightarrow  \ \lambda \ \le \ \min_{v\in V_{w'}(x')}
    \left\{ \frac{\scp{v}{x}+w(v)}{\scp{v}{x'}+w'(v)} \right\} \, .
  \end{align*}
  For each vertex $x$ of $\envelope{w}{P}$ there must be some vertex $x'$ of $\envelope{w'}{P}$
  such that these inequalities hold, and this gives the claim.
\end{proof}

\begin{corollary}
  For two weight function $w$ and $w'$ of $P$ we have
  \[
  \alpha^w_{w'} \ = \ \sup \SetOf{\lambda\ge 0}{(w-\lambda w',\lambda w') \text{ is a
      coherent decomposition of } w} \, .
  \]
\end{corollary}

\begin{corollary}
  If $w$ and $w'$ are weight functions then $\subdivision{w}{P}=\subdivision{w'}{P}$ if
  and only if $\alpha^w_{w'}>0$ and $\alpha^{w'}_w>0$.
\end{corollary}

The set of all regular subdivisions of the convex polytope $P$ is known to have an
interesting structure (see \cite[Chapter 5]{Triangulations} for the details): For a weight function $w\in\RR^n$ of $P$ we consider the set
$S[w]\subset \RR^n$ of all weight functions that are equivalent to~$w$, that is,
\[
S[w] \ := \ \SetOf{x \in \RR^n}{\subdivision{x}{P} = \subdivision{w}{P}} \, .
\]
This set is called the \emph{secondary cone} of $P$ with respect to $w$.  It can be shown
(for instance, see \cite[Corollary 5.2.10]{Triangulations}) that $S[w]$ is indeed a
polyhedral cone and that the set of all $S[w]$ (for all $w$) forms a polyhedral fan
$\SecondaryFan{P}$, called the \emph{secondary fan} of $P$.

It is easily verified that $S[0]$ is the set of all (restrictions of) affine linear
functions and that it is the lineality space of every cone in the secondary fan.  So this fan
can be regarded in the quotient space $\RR^n \Quotient S[0]\cong \RR^{n-d-1}$. If there is
no change for confusion we will identify $w\in\RR^n$ and its image in $\RR^n \Quotient S[0]$.
Furthermore, the secondary fan can be cut with the unit sphere to get a (spherical) polytopal complex on
the set of rays in the fan. This complex carries the same information as the fan itself and will
also be identified with it.

It is a famous result by Gel{\cprime}fand, Kapranov, and Zelevinsky \cite[Theorem 1.7]{GKZ94},
that the secondary fan is the normal fan of a polytope, the \emph{secondary polytope}
$\SecondaryPolytope{P}$ of $P$. This polytope admits a realization as the convex hull of
the so-called \emph{GKZ-vectors} of all (regular) triangulations. The GKZ-vector
$x_\Delta\in \RR^n$ of a triangulation $\Delta$ is defined as ${(x_\Delta)}_v:=\sum_{S}
\vol S$ for all $v\in\Vertices P$, where the sum ranges over all full-dimensional
simplices $S\in\Delta$ which contain~$v$.

A description in terms of inequalities is given by Lee \cite[Section 17.6, Result~4]{Lee97}: The affine hull of $\SecondaryPolytope{P}\subset\RR^n$ is given by the $d+1$ equations
\begin{equation}\label{eq:secPoly-affineHull}
\begin{aligned}
  \sum_{v\in\Vertices P}x_v   \ &= \ (d+1)d\vol P\text{ and}\\
  \sum_{v\in\Vertices P}x_v v \ &= \ ((d+1)\vol P) c_P\,,
\end{aligned}
\end{equation} 
where $c_P$ denotes the centroid of $P$ and $\vol$ denotes the $d$-dimensional volume in
the affine span of $P$, which we can identify with $\RR^d$. The facet defining
inequalities of $\SecondaryPolytope{P}$ are
\begin{align}\label{eq:secPoly-facets}
 \sum_{v\in\Vertices P}w(v)x_v \ \geq \ (d+1) \sum_{Q\in\subdivision{w}{P}}\vol Q \bar w(c_Q)\,,
\end{align}
for all coarsest regular subdivisions $\subdivision{w}{P}$ defined by a weight $w$. Here
$\bar w:P\mapsto \RR$ denotes the piecewise-linear convex function whose graph is given by
the lower facets of $\lift{w}{P}$.

A weight function $w$ such that for all weight functions $w'$ with $\alpha^w_{w'} > 0$ we
have $w'=\lambda w$ (in $\RR^n \Quotient S[0]$) for some $\lambda>0$ is called \emph{prime}.
The set of all prime weight functions for a given polytope P is denoted $\PrimeWeights{P}$.
By this we get directly:

\begin{proposition}\label{prop:primeweights}
  The equivalence classes of prime weights correspond to the extremal rays of the
  secondary fan (and hence to the coarsest regular subdivisions or, equivalently, to the
  facets of the secondary polytope).
\end{proposition}


The following is a reformulation of the fact that the set of all equivalence classes of weight
functions forms a fan (the secondary fan).

\begin{theorem}\label{thm:primeweights}
  Each weight function $w$ on a polytope $P$ can be decomposed into a coherent sum of
  prime weight functions, that is, there are $p_1,\dots,p_k\in\PrimeWeights{P}$ such that
  $w=p_1+\dots+p_k$ is a coherent decomposition.
\end{theorem}

\begin{proof}
  Each weight function $w$ is contained in some cone of the secondary fan of $P$.  Hence
  there are extremal rays $r_1,\dots,r_k$ of the secondary cone and positive real numbers
  $\lambda_1,\dots,\lambda_k$ such that $w=\lambda_1 r_1 + \dots + \lambda_k r_k$; by
  construction, this decomposition is coherent by Lemma \ref{lem:coherent}.  From
  Proposition \ref{prop:primeweights} we know that $p_i:=\lambda_i r_i$ is a prime weight,
  and the claim follows.
\end{proof}

Note that this decomposition is usually not unique.

\section{Splits and the Split Decomposition Theorem}\label{sec:splits}

A \emph{split} $S$ of a polytope $P$ is a decomposition of $P$ without new vertices which
has exactly two maximal cells denoted by $S_+$ and $S_-$. As above, we assume that
$P\subset\RR^{d+1}$ is $d$-dimensional and that $\aff P$ does not contain the origin. Then
the linear span of $S_+\cap S_-$ is a linear hyperplane $H_S$, the \emph{split hyperplane}
of $S$ with respect to $P$. Since $S$ does not induce any new vertices, in particular,
$H_S$ does not meet any edge of $P$ in its relative interior.  Conversely, each hyperplane
which separates $P$ and which does not separate any edge defines a split of $P$.
Furthermore, it is easy to see, that a hyperplane defines a split of $P$ if and only if it
defines a split on all facets of $P$ that it meets in the interior.

The following observation is immediate.  Note that it implies that a hyperplane defines a
split if and only if its does not separate any edge.

\begin{observation}\label{obs:facets}
  A hyperplane that meets $P$ in its interior is a split hyperplane of $P$ if and only if it
  intersects each of its facets $F$ in either a split hyperplane of $F$ or in a face of $F$.
\end{observation}

\begin{remark}\label{rem:oriented-matroid}
  Since the notion of facets and faces of a polytope does only depend on the \emph{oriented
    matroid} of~$P$ it follows from Observation \ref{obs:facets} that the set splits of a
  polytope only depend on the oriented matroid of~$P$. This is in contrast to the fact
  that the set of regular triangulations (see below), in general, depends on the specific
  coordinatization.
\end{remark}

The running theme of this paper is: If a polytope admits sufficiently many splits then
interesting things happen.  However, one should keep in mind that there are many polytopes
without a single split; such polytopes are called \emph{unsplittable}.

\begin{remark}\label{rem:simple}
  If $v$ is a vertex of $P$ such that all neighbors of $v$ in $P$ are contained in a common
  hyperplane $H_v$ then $H_v$ defines a split $S_v$ of $P$.  Such a split is called the
  \emph{vertex split} with respect to~$v$.  For instance, if $P$ is simple then each
  vertex defines a vertex split.

  Since polygons are simple polytopes it follows, in particular, that an unsplittable
  polytope which is not a simplex is at least $3$-dimensional.  An unsplittable
  $3$-polytope has at least six vertices.  An example is a $3$-dimensional cross polytope
  whose vertices are perturbed into general position.
\end{remark}

\begin{proposition}
  Each $2$-neighborly polytope is unsplittable.
\end{proposition}

\begin{proof}
  Assume that $S$ is a split of $P$, and $P$ is $2$-neighborly.  Recall that the latter
  property means that any two vertices of $P$ are joined by an edge.  Choose vertices
  $v\in S_+\setminus S_-$ and $w\in S_-\setminus S_+$. Then the segment $[v,w]$ is an edge
  of $P$ which is separated by the split hyperplane $H_S$.  This is a contradiction to the
  assumption that $S$ was a split of $P$.
\end{proof}

It is clear that splits yield coarsest subdivisions; but the following lemma says that
they even define facets of the secondary polytope.

\begin{lemma}\label{lem:regular}
  Splits are regular.
\end{lemma}

\begin{proof}
  Let $S$ be a split of $P$.  We have to show that $S$ is induced by a weight function.
  Let $a$ be a normal vector of the split hyperplane $H_S$.  We define $w_S:\Vertices(P)\to\RR$ by
  \begin{equation}\label{eq:split:lifting}
    w_S(v) \ := \
    \begin{cases}
      \abs{av} & \text{if $v\in S_+$ \, ,} \\
      0 & \text{if $v\in S_-$ \, .}
    \end{cases}
  \end{equation}
  Note that this function is well-defined since for $v\in H_S=S_+\cap S_-$ we have
  $av=0$.  It is now obvious that $w$ induces the split $S$ on $P$.
\end{proof}

\begin{example}
  In Example~\ref{ex:hexagon:1} the three weight functions $w_1$, $w_2$, $w_3$ define
  splits of the hexagon $H$.
\end{example}

By specializing Equation~\eqref{eq:secPoly-facets}, a facet defining inequality for the
split $S$ is given by
\begin{align}\label{eq:split-equation}
\sum_{v\in \Vertices(P\cap S_+)}|av|x_v \ \geq \ |ac_{P\cap S_+}|(d+1) \vol(P\cap S_+)\,.
\end{align}
Note that $a$ is a normal vector of the split hyperplane $H_S$ as above, and $c_{P\cap
  S_+}$ is the centroid of the polytope $P\cap S_+$.  By taking the inequalities
\eqref{eq:split-equation} for all splits $S$ of $P$ together with the equations
\eqref{eq:secPoly-affineHull} we get an $(n-d-1)$-dimensional polyhedron $\SplitPoly{P}$
which we will call the \emph{split polyhedron} of $P$. Obviously, we have
$\SecondaryPolytope{P}\subseteq\SplitPoly{P}$ so the split polyhedron can be seen as an outer
``approximation'' of the secondary polytope. In fact, by Remark
\ref{rem:oriented-matroid}, $\SplitPoly{P}$ is a common ``approximation'' for the
secondary polytopes of all possible coordinatizations of the oriented matroid of $P$. If
$P$ has sufficiently many splits the split polyhedron is bounded; in this case
$\SplitPoly{P}$ is called the \emph{split polytope} of $P$.

One can show that each simple polytope has a bounded split polyhedron.  Here we give
two examples.

\begin{example}\label{ex:n-gon:1}
  Let $P$ be a an $n$-gon for $n\ge 4$.  Then each pair of non-neighboring vertices
  defines a split of~$P$.  Each triangulation is regular and, moreover, a split
  triangulation.

  The secondary polytope of $P$ is the associahedron $\Assoc_{n-3}$, which is a simple
  polytope of dimension $n-3$.  Since the only coarsest subdivisions of $P$ are the splits
  it follows that the split polytope of $P$ coincides with its secondary polytope.
\end{example}

\begin{example}\label{ex:3cube:1}
  The $74$ triangulations of the regular $3$-cube $C_3=[-1,1]^3$ are all regular, and $26$
  of them are induced by splits.  The total number of splits is $14$: There are eight
  vertex splits ($C$ being simple) and six splits defined by parallel pairs of diagonals
  in an opposite pair of cube facets.  The secondary polytope of $C$ is a $4$-polytope
  with $f$-vector $(74,152,100,22)$; see Pfeifle~\cite{Pfeifle00} for a complete
  description.

  The split polytope of $C_3$ is neither simplicial nor simple and has the $f$-vector
  $(22,60,52,14)$.  A Schlegel diagram is shown in Figure~\ref{fig:cube:split}.
\end{example}

\begin{example}
  There are nearly $88$ million regular triangulations of the $4$-cube $C_4=[-1,1]^4$ that
  come in $235,277$ equivalence classes.  The $4$-cube has four different types of splits:
  The vertex splits, the split obtained by cutting with $H:=\smallSetOf{x}{\sum x_i = 0}$
  (and its images under the symmetry group of the cube), and, finally, two kinds of splits
  induced by the two kinds of splits of the $3$-cube.  The split obtained from the vertex
  split of the $3$-cube is the one discussed in \cite[Example 20 (The missing
  split)]{Grier+08}. See also \cite{Grier+08} for a complete discussion of the secondary
  polytope of $C_4$.  Examples of triangulations of the $4$-cube that are induced by
  splits include the first two in \cite[Example~10 \& Figure~3]{Grier+08} and the one
  shown in Figure~\ref{fig:4-cube}.
\end{example}

\begin{figure}[htb]
  \includegraphics[width=.4\textwidth]{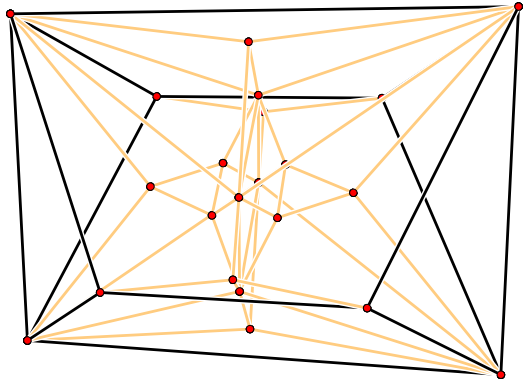}
  \caption{Schlegel diagram of the split polytope of the regular $3$-cube.}
  \label{fig:cube:split}
\end{figure}

A weight function $w$ on a polytope $P$ is called \emph{split prime} if for all splits $S$ of
$P$ we have $\alpha^w_{w_S}=0$.  The following can be seen as a generalization of Bandelt
and Dress \cite[Theorem 3]{BandeltDress92}, and as a reformulation of Hirai's Theorem 2.2
\cite{Hirai06}.

\begin{theorem}[Split Decomposition Theorem]\label{thm:splitdecomposition}
  Each weight function $w$ has a coherent decomposition
  \begin{equation}\label{eq:splitdecomposition}
  w \ = \ w_0 + \sum_{\text{$S$ split of $P$}} \lambda_S w_S \, ,
  \end{equation}
  where $w_0$ is split prime, and this is unique among all coherent decompositions of~$w$.
\end{theorem}

This is called the \emph{split decomposition} of $w$.

\begin{proof}
  We first consider the special case where the subdivision $\subdivision{w}{P}$ induced by
  $w$ is a common refinement of splits.  Then each face $F$ of codimension~$1$ in
  $\subdivision{w}{P}$ defines a unique split $S(F)$, namely the one with split hyperplane
  $H_{S(F)}=\lin F$.  Moreover, whenever $S$ is an arbitrary split of $P$ then
  $\alpha^w_{w_S}>0$ if and only if $H_S\cap P$ is a face of $\Delta_w$ of
  codimension~$1$.  This gives a coherent subdivision $w = \sum_S \alpha^w_{w_S} w_S$,
  where the sum ranges over all splits $S$ of $P$.  Note that the uniqueness follows
  from the fact that for each codimension-$1$-faces of $\Delta_w$ there is a unique split
  which coarsens it.

  For the general case, we let
  \[
  w_0 \ := \  w - \sum_{\text{$S$ split of $P$}} \alpha^{w}_{w_S} w_S \, .
  \]
  By construction, $w_0$ is split prime, and the uniqueness of the split decomposition of $w$
  follows from the uniqueness of the split decomposition of $w-w_0$.
\end{proof}

In fact, the sum in \eqref{eq:splitdecomposition} only runs over all splits in
$\cS(w):=\smallSetOf{w_S}{\alpha^w_{w_S}>0}$.  The uniqueness part of the theorem gives us
the following interesting corollary (see also Bandelt and Dress \cite[Corollary 5]{BandeltDress92}, and Hirai \cite[Proposition 3.6]{Hirai06}):

\begin{corollary}\label{cor:split-linear-independence}
  For a weight function $w$ the set $\cS(w)\cup\{w_0\}$ is linearly independent.  In
  particular, $\#\cS(w)\leq n-d-1$, if $\#\cS(w)= n-d-1$ then $w_0=0$, and if $\#\cS(w)=
  n-d-2$ then $w_0$ is a prime weight function.
\end{corollary}

\begin{proof}
  Suppose the set would be linearly dependent. This would yield a relation
  \[
  \sum_{S\in \cS} \lambda_S w_{S} \ = \ \lambda_0 w_0 + \sum_{S \in \cS(w)\setminus \cS} \lambda_S w_{S}
  \]
  with coefficients $\lambda_0,\lambda_S\geq 0$ for some $\cS\subset\cS(w)$. However, this
  contradicts the uniqueness part of Theorem~\ref{thm:splitdecomposition} for the weight
  function $w':= \sum_{S\in\cS} \lambda_S w_{S}$.
  
  The cardinality constraints now follow from the fact that the weight functions live in
  $\RR^n \Quotient S[0]\cong \RR^{n-d-1}$.
\end{proof}

The next lemma is a specialization of Corollary~\ref{cor:coherent} to the case of splits and
their weight functions.

\begin{lemma}\label{lem:weaklycompatible}
  Let $\cS$ be a set of splits for $P$.  Then the following statements are equivalent.
  \begin{enumerate}
  \item The corresponding decomposition $w:=\sum_{S\in\cS} w_S$ is coherent,
  \item there exists a common refinement of all $S\in\cS$ (induced by $w$),
  \item there is a regular triangulation of $P$ which refines all $S\in \cS$.
  \end{enumerate}
\end{lemma}

Instead of ``set of splits'' we equivalently use the term \emph{split system}. A split
system is called \emph{weakly compatible} if one of the properties of Lemma
\ref{lem:weaklycompatible} is satisfied. Moreover, two splits $S_1$ and $S_2$ such that
$H_{S_1}\cap H_{S_2}$ does not meet $P$ in its interior are called \emph{compatible}. This
notion generalizes to arbitrary split systems in different ways: A set $\cS$ of splits is
called compatible if any two of its splits are compatible. It is \emph{incompatible} if it
is not compatible, and it is \emph{totally incompatible} if any two of its splits are
incompatible.  It is clear that total incompatibility implies incompatibility, and that
compatibility implies weak compatibility (but the converse does not hold, see Example~\ref{ex:3cube:2}).

For an arbitrary split system $\cS$ we define its weight function as
\[
w_\cS \ := \ \sum_{S\in\cS}w_S \, .
\]
If $\cS$ is weakly compatible then $\subdivision{\cS}{P}:=\subdivision{w_\cS}{P}$ is the
coarsest subdivision refining all splits in $\cS$. We further abbreviate
$\envelope{\cS}{P}:=\envelope{w_\cS}{P}$ and $\tightspan{\cS}{P}:=\tightspan{w_\cS}{P}$.

\begin{remark}
  The split decomposition \eqref{eq:splitdecomposition} of a weight function $w$ of the
  $d$-polytope $P$ can actually be computed using our formula \eqref{eq:coherency-index}.
  Provided we already know the, say, $t$ vertices of the tight span of $w$ and the, say,
  $s$ splits of $P$, this takes $O(s \, t \, d \, n)$ arithmetic operations over the reals
  (or the rationals), where $n=\#\Vertices{P}$.
\end{remark}

\section{Split Complexes and Split Subdivisions}

Let $P$ be a fixed $d$-polytope, and let $\cS(P)$ be the set of all splits of $P$.  The
notions of compatibility and weak compatibility of splits give rise to two abstract
simplicial complexes with vertex set $\cS(P)$.  We denote them by $\SplitComplex{P}$ and
$\WeakSplitComplex{P}$, respectively.  Since compatibility implies weak compatibility
$\SplitComplex{P}$ is a subcomplex of $\WeakSplitComplex{P}$.  Moreover, if $\cS\subseteq
\cS(P)$ is a split system such that any two splits in $\cS$ are compatible then the whole
split system $\cS$ is compatible.  This can also be phrased in graph theory language: The
compatibility relation among the splits defines an undirected graph, whose cliques
correspond to the faces of $\SplitComplex{P}$.  In particular, we have the following:

\begin{proposition}
  The split complex $\SplitComplex{P}$ is a flag simplicial complex.
\end{proposition}

Note that we did not assume that $P$ admits any split.  If $P$ is unsplittable then the
(weak) split complex of $P$ is the void complex $\emptyset$.

Theorem \ref{thm:splitdecomposition} tells us that the fan spanned by the rays that induce
splits is a simplicial fan contained in (the support of) $\SecondaryFan{P}$.  This fan was
called the \emph{split fan} of $P$ by Koichi \cite{Koichi06}.  Denoting by
$\SecondaryFanOne{P}$ the (spherical) polytopal complex which arises from
$\SecondaryFan{P}$ by intersecting with the unit sphere, this leads to the following
observation:

\begin{corollary}\label{cor:split-secondary-complex}
  The simplicial complex $\SplitComplex{P}$ is a subcomplex of the polytopal complex
  $\SecondaryFanOne{P}$.
\end{corollary}

\begin{proof}
  The tight span of a compatible system $\cS$ of splits of $P$ is a tree by
  Proposition~\ref{prop:tree}.  This implies that the cell $C$ in $\SecondaryFanOne{P}$
  generated by $\cS$ does not contain vertices whose tight span is of dimension greater
  than one.  Thus the vertices of $C$ are precisely the splits in $\cS$.
\end{proof}

\begin{remark}
  The weak split complex of $P$ is usually not a subcomplex of $\SecondaryFanOne{P}$; see
  Example~\ref{ex:3cube:2}.  However, one can show that $\WeakSplitComplex{P}$ is homotopy
  equivalent to a subcomplex of $\SecondaryFanOne{P}$.
\end{remark}

From Corollary \ref{cor:split-linear-independence} we can trivially derive an upper bound on
the dimensions of the split complex and the weak split complex. This bound is sharp for
both types of complexes as we will see in Example~\ref{ex:n-gon:2} below.

\begin{proposition}\label{prop:dim-split-complex}
  The dimensions of $\SplitComplex{P}$ and $\WeakSplitComplex{P}$ are bounded from above by $n-d-2$.
\end{proposition}

A regular subdivision (triangulation) $\Delta$ of $P$ is called a \emph{split subdivision
  (triangulation)} if it is the common refinement of a set $\cS$ of splits of $P$.
Necessarily, the split system $\cS$ is weakly compatible, and $\cS$ is a face of
$\WeakSplitComplex{P}$.  Conversely, all faces of $\WeakSplitComplex{P}$ arise in this
way.

\begin{corollary}\label{cor:weakly-compatible-triangulation}
  If $\cS$ is a facet of $\WeakSplitComplex{P}$ with $\#\cS=n-d-1$ then the split subdivision
  $\subdivision{\cS}{P}$ is a split triangulation.
\end{corollary}

\begin{proof}
  Corollary \ref{cor:split-linear-independence} implies that $W:=\smallSetOf{w_S}{S\in\cS}$ is
  linearly independent and hence a basis of $\RR^n \Quotient S[0]\cong \RR^{n-d-1}$. So the
  cone spanned by $W$ is full-dimensional and hence corresponds to a vertex of the secondary
  polytope.
\end{proof}

The following is a characterization of the faces of $\SplitComplex{P}$, and it says that
split complexes are always ``spaces of trees''.

\begin{proposition}[Hirai \cite{Hirai06}, Proposition 2.9]\label{prop:tree}
  Let $\cS$ be a split system on $P$.  Then the following statements are equivalent.
  \begin{enumerate}
  \item\label{it:tree:compatible} $\cS$ is compatible,
  \item\label{it:tree:1D} $\tightspan{\cS}{P}$ is $1$-dimensional, and
  \item\label{it:tree:tree} $\tightspan{\cS}{P}$ is a tree.
  \end{enumerate}
\end{proposition}

\begin{proof}
  Assume that $\subdivision{\S}{P}$ is induced by the compatible split system $\cS\ne\emptyset$.  By
  definition, for any two distinct splits $S_1,S_2\in\cS$ the hyperplanes $H_{S_1}$ and $H_{S_2}$
  do not meet in the interior of~$P$.  This implies that there are no
  interior faces in $\subdivision{\S}{P}$ of codimension greater than $1$.  By 
  Proposition~\ref{prop:duality},
  this says that $\dim\tightspan{\S}{P}\le 1$. Since $\cS\ne\emptyset$ we have that
  $\dim\tightspan{\S}{P}=1$.  Thus \eqref{it:tree:compatible} implies \eqref{it:tree:1D}.

  The statement \eqref{it:tree:tree} follows from \eqref{it:tree:1D} as the tight span is contractible.

  Suppose that $\tightspan{\S}{P}$ is a tree.  Then each edge is dual to a split
  hyperplane.  The system $\cS$ of all these splits is clearly weakly compatible since it
  is refined by $\subdivision{\S}{P}$.  Assume that there are splits $S_1,S_2\in\cS$ such
  that the corresponding split hyperplanes $H_{S_1}$ and $H_{S_2}$ meet in the interior
  of~$P$.  Then $H_{S_1}\cap H_{S_2}$ is an interior face in $\subdivision{\S}{P}$ of
  codimension~$2$, contradicting our assumption that $\tightspan{\S}{P}$ is a tree.  This
  proves \eqref{it:tree:compatible}, and hence the claim follows.
\end{proof}

\begin{remark}
  A $d$-dimensional polytope is called \emph{stacked} if it has a triangulation in which
  there are no interior faces of dimension less than $d-1$. So it follows from Proposition~\ref{prop:tree}
  that a polytope is stacked if and only if there exists a split triangulation induced by
  a compatible system of splits.
\end{remark}
  
\begin{example}\label{ex:n-gon:2}
  Let $P$ be a an $n$-gon for $n\ge 4$.  As already pointed out in
  Example~\ref{ex:n-gon:1}, each pair of non-neighboring vertices defines a split of $P$.
  Two such splits are compatible if and only if they are weakly compatible.

  The secondary polytope of $P$ is the associahedron $\Assoc_{n-3}$, and the split complex
  of $P$ is isomorphic to the boundary complex of its dual.  In particular,
  $\SplitComplex{P}=\WeakSplitComplex{P}$ is a pure and shellable simplicial complex of
  dimension $n-4$, which is homeomorphic to $\Sph^{n-4}$. This shows that the bound in
  Proposition~\ref{prop:dim-split-complex} is sharp.  From Catalan combinatorics it is
  known that the (split) triangulations of $P$ correspond to the binary trees on $n-2$
  nodes; see~\cite[Section~1.1]{Triangulations}.
\end{example}

\begin{example}\label{ex:cross_d}
  The splits of the regular cross polytope $X_d=\conv\{\pm e_1,\pm e_2,\dots,\pm e_d\}$ in
  $\RR^d$ are induced by the $d$ reflection hyperplanes $x_i=0$.  Any $d-1$ of them are
  weakly compatible and define a triangulation of $X_d$ by Corollary
  \ref{cor:weakly-compatible-triangulation}.  (Of course, this can also be seen directly.)
  All triangulations of $X_d$ arise in this way.  This shows that $\WeakSplitComplex{X_d}$
  is isomorphic to the boundary complex of a $(d-1)$-dimensional simplex, which is also
  the secondary polytope and the split polytope of $X_d$.  Any two reflection hyperplanes
  meet in the interior of $X_d$, whence no two splits are compatible.  This says that
  $\SplitComplex{X_d}$ consists of $d$ isolated points.
\end{example}

\begin{example}\label{ex:3cube:2}
  As we already discussed in Example~\ref{ex:3cube:1} the regular $3$-cube $C_3=[-1,1]^3$
  has a total number of $14$ splits. The split complex $\SplitComplex{C}$ is
  $3$-dimensional but not pure; its $f$-vector reads $(14,40,32,2)$.  The two
  $3$-dimensional facets correspond to the two non-unimodular triangulations of $C$
  (arising from splitting every other vertex).  The reduced homology is concentrated in
  dimension two, and we have $H_2(\SplitComplex{C_3};\ZZ)\cong\ZZ^3$.  The graph
  indicating the compatibility relation among the splits is shown in
  Figure~\ref{fig:SplitComplex-3-cube}.

  Figure~\ref{fig:3cube-triag} shows three triangulations of $C_3$.  The left one is
  generated by a totally incompatible system of three splits; that is, it is a facet of
  $\WeakSplitComplex{C_3}$ which is not a face of $\SplitComplex{C_3}$.  The right one is (not
  unimodular and) generated by a compatible split system (of four vertex splits); that is,
  it is a facet of both $\SplitComplex{C_3}$ and $\WeakSplitComplex{C_3}$.  The middle one is
  not generated by splits at all.
  
  The triangulation $\Delta$ on the left uses only three splits.  This examples shows that
  the converse of Corollary~\ref{cor:weakly-compatible-triangulation} is not true, that
  is, a weakly compatible split system that defines a triangulation does not have to be
  maximal with respect to cardinality. Furthermore, the triangulation $\Delta$ can also be
  obtained as the common refinement of two non-split coarsest subdivisions. The cell in
  $\SecondaryFanOne{C_3}$ corresponding to $\Delta$ is a bipyramid over a triangle. The
  vertices of this triangle (which is not a face of $\SecondaryFanOne{C_3}$) correspond to
  the three splits, so the relevant cell in $\WeakSplitComplex{C_3}$ is a triangle, and
  the apices corresponds to the non-split coarsest subdivisions mentioned above.  Since
  the three splits are totally incompatible there does not exist a corresponding face in
  $\SplitComplex{C_3}$, and the intersection with $\SplitComplex{C_3}$ consists of three
  isolated points.
\end{example}

\begin{figure}[thb]
  \includegraphics[width=.35\textwidth]{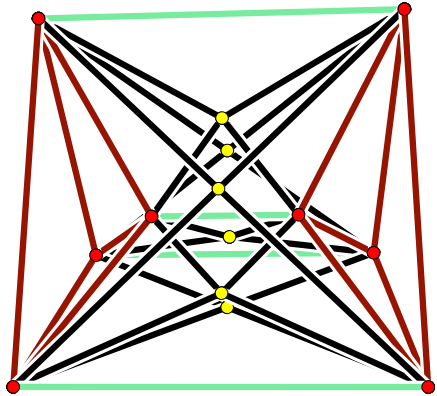}
  \caption{Compatibility graph of the splits of the regular $3$-cube. The four (red) nodes
    to the left and the four (red) nodes to the right correspond to the vertex splits.}
  \label{fig:SplitComplex-3-cube}
\end{figure}

%

A polytopal complex is \emph{zonotopal} if each face is zonotope.  A \emph{zonotope} is
the Minkowski sum of line segments or, equivalently, the affine projection of a regular
cube.  Any graph, that is, a $1$-dimensional polytopal complex, is zonotopal in a trivial
way. So especially tight spans of splits and, by Proposition~\ref{prop:tree}, of
compatible splits systems are zonotopal. In fact, this is even true for arbitrary weakly
compatible splits systems. See also Bolker \cite[Theorem~6.11]{Bolker69} and Hirai
\cite[Corollary 2.8]{Hirai06}.

\begin{theorem}\label{thm:zonotopal}
  Let $\cS$ be a weakly compatible split system on~$P$.  Then the tight span
  $\tightspan{\cS}{P}$ is a (not necessarily pure) zonotopal complex.
\end{theorem}

\begin{proof}
  Let $F$ be a face of $\tightspan{\cS}{P}$.  Since by Lemma~\ref{lem:weaklycompatible}
  we have that $\envelope{\cS}{P}=\sum_{S\in\cS} \envelope{w_S}{P}$ we get (by the same arguments used in
  the proof of Lemma~\ref{lem:coherent}) that $F=\sum_{S\in\cS} F_S$ for
  faces $F_S$ of $\tightspan{w_S}{P}$. The claim now follows from the fact that
  $\tightspan{w_S}{P}$ is a line segment for all $S\in\cS$.
\end{proof}

A triangulation of a $d$-polytope is \emph{foldable} if its vertices can be colored with
$d$ colors such that each edge of the triangulation receives two distinct colors.  This
is equivalent to requiring that the dual graph of the triangulation is bipartite; see
\cite[Corollary 11]{ProjSimplePolytopes}.  Note that foldable simplicial complexes are
called ``balanced'' in~\cite{ProjSimplePolytopes}.  The three triangulations of the
regular $3$-cube in Figure~\ref{fig:3cube-triag} are foldable.

\begin{figure}[hbt]
  \includegraphics[width=.3\textwidth]{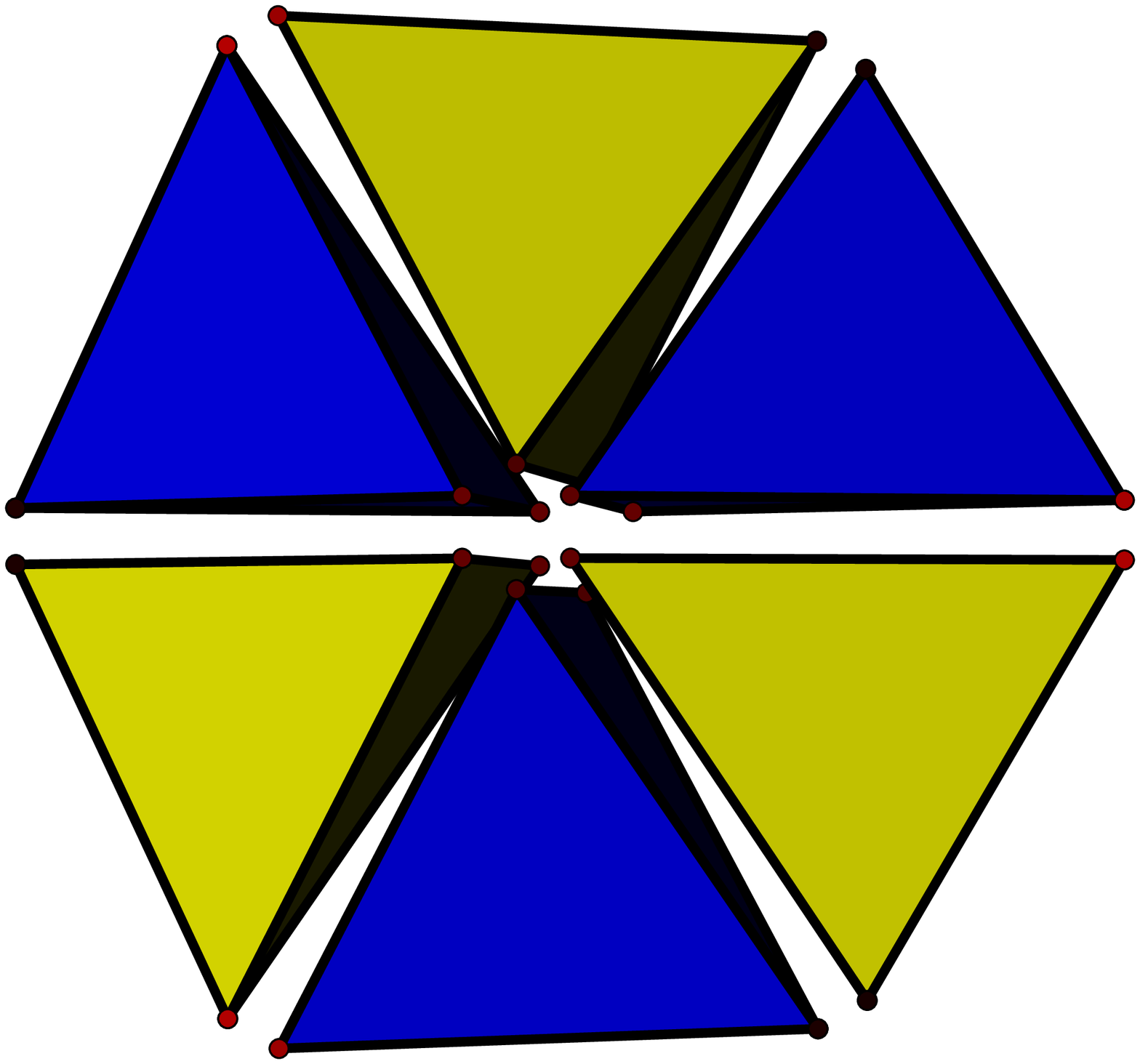}\hfill
  \includegraphics[width=.3\textwidth]{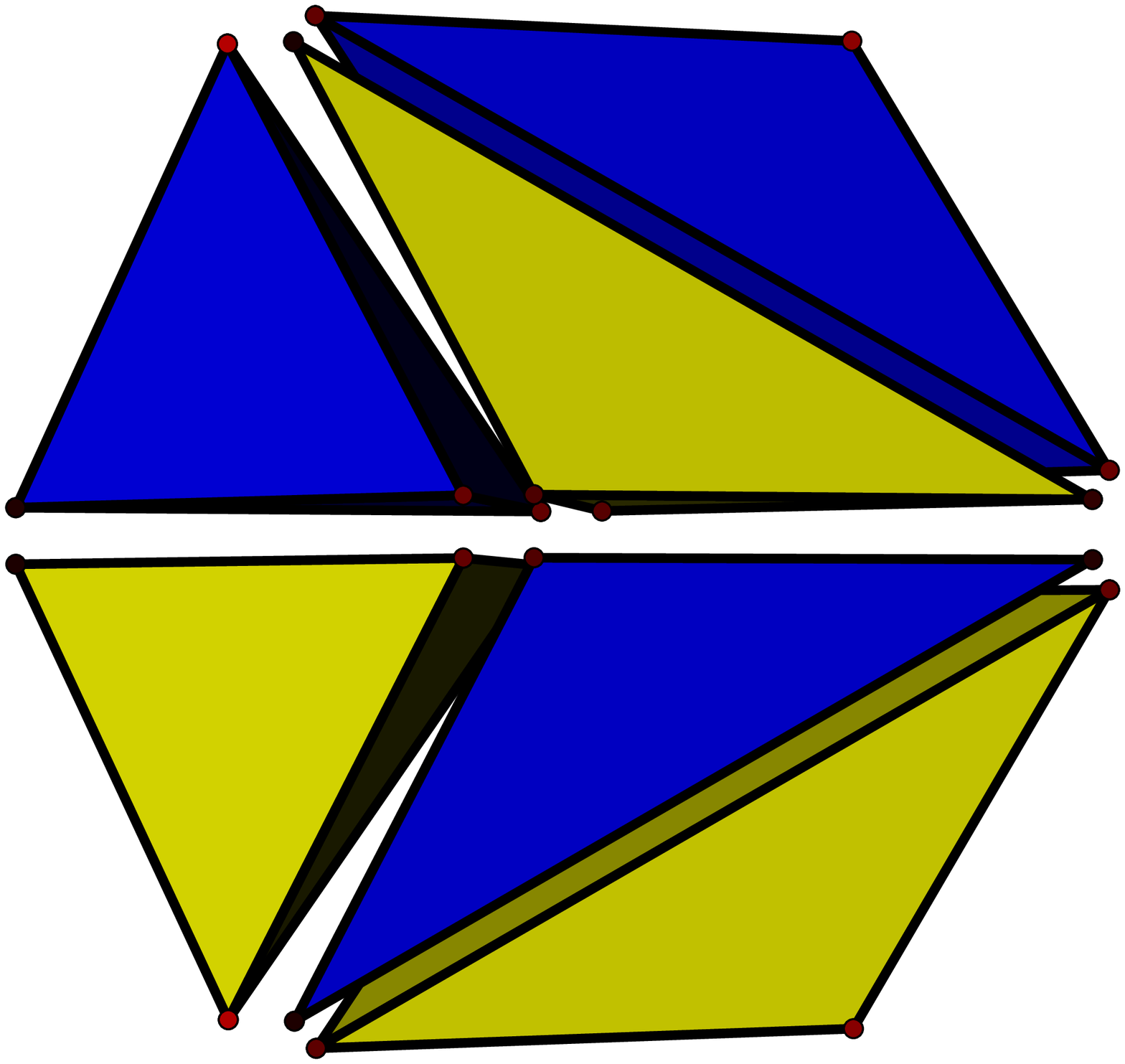}\hfill
  \includegraphics[width=.3\textwidth]{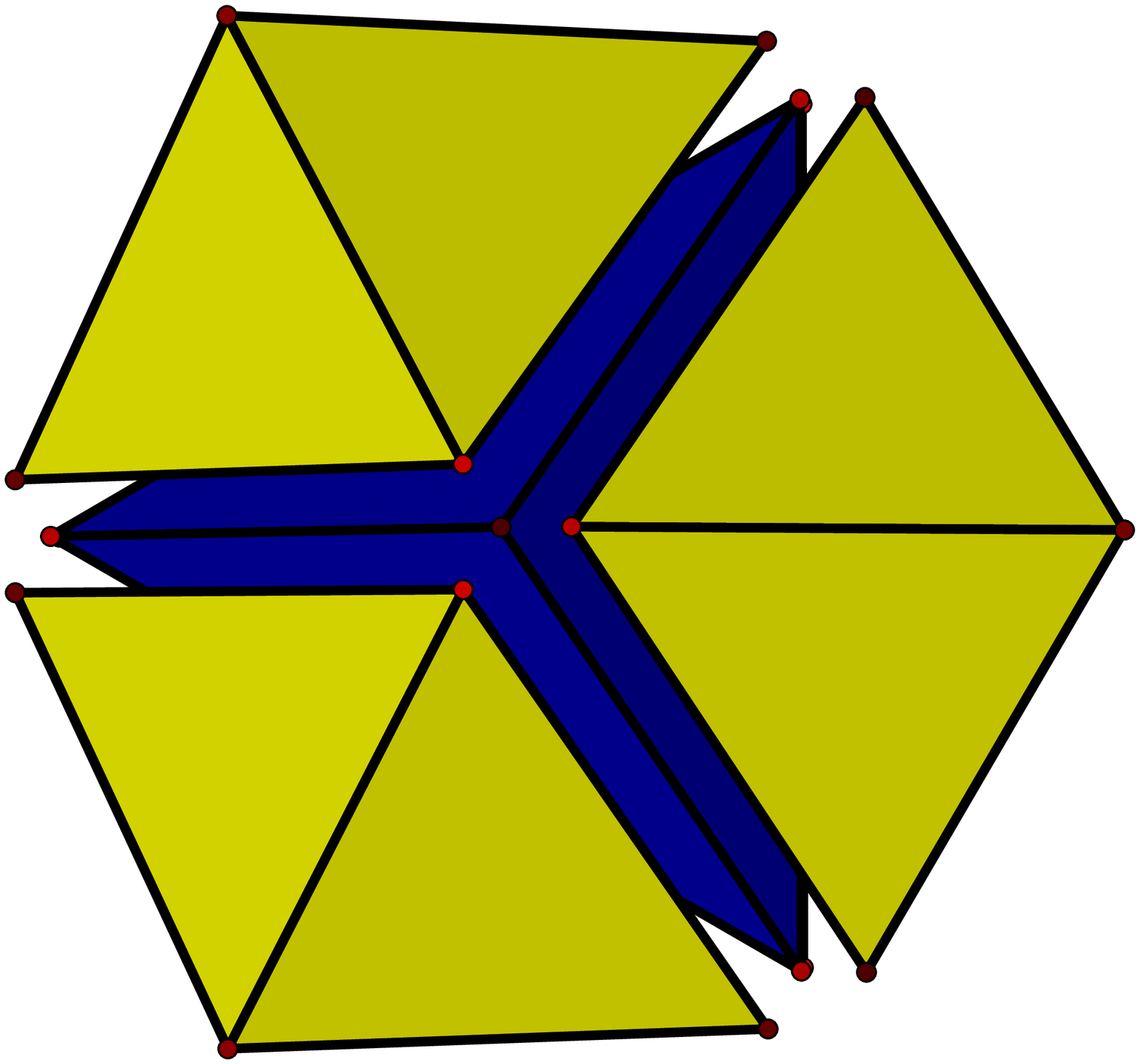}
  \caption{Three foldable triangulations of the regular $3$-cube.}
  \label{fig:3cube-triag}
\end{figure}

\begin{corollary}\label{cor:foldable}
  Each split triangulation is foldable.
\end{corollary}

\begin{proof}
  Let $\cS$ be a weakly compatible split system such that $\subdivision{\cS}{P}$ is a
  triangulation.  By Theorem \ref{thm:zonotopal} each $2$-dimensional face of the
  tight span $\tightspan{\cS}{P}$ has an even number of vertices.  This implies that
  $\subdivision{\cS}{P}$ is a triangulation of $P$ such that each of its interior
  codimension-$2$-cell is contained in an even number of maximal cells.  Now the claim follows
  from \cite[Corollary 11]{ProjSimplePolytopes}.
\end{proof}

\begin{example}
  Let $C_4$ be the $4$-dimensional cube. In Figure~\ref{fig:4-cube} there is a picture of
  the tight span $\tightspan{\cS}{C_4}$ of a split system $\cS$ of $C_4$ with 10 weakly
  compatible splits.  As proposed by Theorem~\ref{thm:zonotopal} the complex is zonotopal.
  It is $3$-dimensional and its $f$-vector reads $(24,36,14,1)$.  The number of vertices
  equals $24=4!$ which is the normalized volume of $C_4$, and hence $\subdivision{\cS}{C_4}$ is,
  in fact, a triangulation. By Corollary \ref{cor:foldable} this triangulation is
  foldable.
\end{example}

\begin{figure}[hbt]
  \includegraphics[width=.5\textwidth]{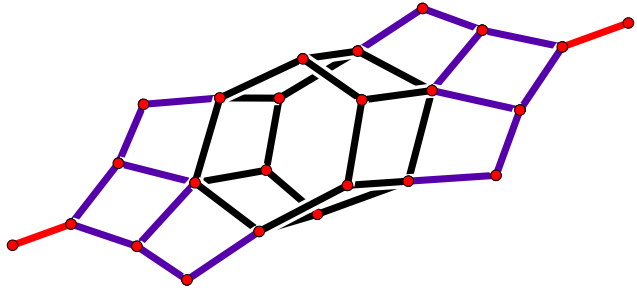}
  \caption{The tight span of a split triangulation of the $4$-cube.}
  \label{fig:4-cube}
\end{figure}

\section{Hypersimplices}\label{sec:hypersimplices}

As a notational shorthand we abbreviate $[n]:=\{1,2,\dots,n\}$ and
$\binom{[n]}{k}:=\smallSetOf{X\subseteq[n]}{\# X=k}$.  The $k$-th \emph{hypersimplex} in
$\RR^n$ is defined as
\[
\Hypersimplex{k}{n} \ := \ \SetOf{x\in[0,1]^n}{\sum_{i=1}^n x_i=k} \ = \ \conv\SetOf{\sum_{i\in A}e_i}{A\in\binom{[n]}{k}} \, .
\]
It is $(n-1)$-dimensional and satisfies the conditions of Section~\ref{sec:coherency}.
Throughout the following we assume that $n\ge 2$ and $1\le k\le n-1$.

A hypersimplex $\Hypersimplex{1}{n}$ is an $(n-1)$-dimensional simplex.  For arbitrary
$k\ge 1$ we have $\Hypersimplex{k}{n}\cong\Hypersimplex{n-k}{n}$.  Moreover, for $p\in[n]$
the equation $x_p=0$ defines a facet isomorphic to $\Hypersimplex{k}{n-1}$.  And, if $k\ge
2$, the equation $x_p=1$ defines a facet isomorphic to $\Hypersimplex{k-1}{n}$.  This list
of facets (induced by the facets of $[0,1]^n$) is exhaustive.  Since the hypersimplices
are not full-dimensional, the facet defining (affine) hyperplanes are not unique.
For the following it will be convenient to work with linear hyperplanes.  This way $x_p=1$
gets replaced by
\begin{equation}\label{eq:hypersimplex:facet}
  (k-1) x_p \ = \ \sum_{i\in [n]\setminus\{p\}}x_i \, .
\end{equation}

The triplet $(A,B;\mu)$ with $\emptyset\ne A,B\subsetneq[n]$, $A\cup B=[n]$, $A\cap
B=\emptyset$ and $\mu\in\NN$ defines the linear equation
\begin{equation}\label{eq:hypersimplex:split}
  \mu \sum_{i\in A} x_i \ = \ (k-\mu) \sum_{i\in B} x_i \, .
\end{equation}
The corresponding (linear) hyperplane in $\RR^n$ is called the
\emph{$(A,B;\mu)$-hyperplane}.  Clearly, $(A,B;\mu)$ and $(B,A;k-\mu)$ define the same
hyperplane.  The Equation~\eqref{eq:hypersimplex:facet} corresponds to the
$(\{p\},[n]\setminus\{p\};k-1)$-hyperplane.

\begin{lemma}\label{lem:hypersimplex:splits}
  The \emph{$(A,B;\mu)$-hyperplane} is a split hyperplane of $\Hypersimplex{k}{n}$ if and
  only if $k-\mu+1\le\#A\le n-\mu-1$ and $1\le\mu\le k-1$.
\end{lemma}

\begin{proof} 
  It is clear that the $(A,B;\mu)$-hyperplane does not meet the interior of
  $\Hypersimplex{k}{n}$ if $\mu\le 0$ or if $\mu\ge k$. Especially, we may assume
   that $k\ge 2$.
   
  Suppose now that $\#A\le k-\mu$.  Then each point $x\in\Hypersimplex{k}{n}$ satisfies
  $\sum_{i\in A}x_i\le k-\mu$ and $\sum_{i\in B}x_i\ge k-(k-\mu)=\mu$.  This implies that
  $\mu \sum_{i\in A} x_i \le (k-\mu) \sum_{i\in B} x_i$, which says that all points in
  $\Hypersimplex{k}{n}$ are contained in one of the two halfspaces defined by the
  $(A,B;\mu)$-hyperplane.  Hence it does not define a split.  A similar argument shows
  that $\#A\le n-\mu-1$ is necessary in order to define a split.

  Conversely, assume that $k-\mu+1\le\#A\le n-\mu-1$ and $1\le\mu\le k-1$.  We define a
  point $x\in\Hypersimplex{k}{n}$ by setting
  \[
  x_i \ := \
  \begin{cases}
    \frac{k-\mu}{\# A} & \text{if $i\in A$} \\
    \frac{\mu}{\# B} & \text{if $i\in B$} \, .
  \end{cases}
  \]
  Since $0<\frac{k-\mu}{\# A}<1$ and $0<\frac{\mu}{\# B}<1$ the point $x$ is contained in
  the (relative) interior of $\Hypersimplex{k}{n}$.  Moreover, $x$ satisfies the
  Equation~\eqref{eq:hypersimplex:split}, and so the $(A,B;\mu)$-hyperplane passes through
  the interior of $\Hypersimplex{k}{n}$.

  It remains to show that the $(A,B;\mu)$-hyperplane does not separate any edge.  Let $v$
  and $w$ be two adjacent vertices.  So we have some $\{p,q\}\in\tbinom{[n]}{2}$ with
  $v-w=e_p-e_q$.  Aiming at an indirect argument, we assume that $v$ and $w$ are on
  opposite sides of the $(A,B;\mu)$-hyperplane, that is, without loss of generality
  $\mu\sum_{i\in A}v_i>(k-\mu)\sum_{i\in B}v_i$ and $\mu\sum_{i\in A}w_i<(k-\mu)\sum_{i\in
    B}w_i$.  This gives
  \[
  0 \ < \ \mu\sum_{i\in A}v_i \, - \, (k-\mu)\sum_{i\in B} v_i \ = \ \mu(\chi_A(p)-\chi_A(q))
  \]
  and
  \[
  0 \ < \ (k-\mu)\sum_{i\in B}w_i \, - \, \mu\sum_{i\in A}w_i \ = \ (k-\mu)(\chi_B(p)-\chi_B(q))  \, ,
  \]
  where characteristic functions are denoted as $\chi_{\cdot}(\cdot)$.  Since $\mu>0$ and
  $\mu<k$ it follows that $\chi_A(q)<\chi_A(p)$ and $\chi_B(q)<\chi_B(p)$.  Now the
  characteristic functions take values in $\{0,1\}$ only, and we arrive at
  $\chi_A(q)=\chi_B(q)=0$ and $\chi_A(p)=\chi_B(p)=1$.  Both these equations contradict
  the fact that $(A,B)$ is a partition of $[n]$.  So we conclude that, indeed, the
  $(A,B;\mu)$-hyperplane defines a split.
\end{proof}

This allows to characterize the splits of the hypersimplices.

\begin{proposition}\label{prop:hypersimplex:splits}
  Each split hyperplane of $\Hypersimplex{k}{n}$ is defined by a linear equation of the
  type~\eqref{eq:hypersimplex:split}.
\end{proposition}

\begin{proof}
  Using Observation~\ref{obs:facets} and exploiting the fact that facets of hypersimplices
  are hypersimplices we can proceed by induction on $n$ and $k$ as follows.

  Our induction is based on the case $k=1$.  Since $\Hypersimplex{1}{n}$ is an
  $(n-1)$-simplex, which does not have any splits, the claim is trivially satisfied.  The
  same holds for $k=n-1$ as $\Hypersimplex{n-1}{n}\cong\Hypersimplex{1}{n}$.

  For the rest of the proof we assume that $2\le k\le n-2$.  In particular, this implies
  that $n\ge 4$.

  Let $\sum_{i\in[n]}\alpha_i x_i=0$ define a split hyperplane $H$ of
  $\Hypersimplex{k}{n}$. The facet defining hyperplane $F_p=\smallSetOf{x}{x_p=0}$ is
  intersected by $H$, and we have
  \[
  F_p\cap H \ = \ \biggl\{x\in\RR^n\,\biggl|\,\sum_{i\in[n]\setminus\{p\}}\alpha_ix_i=0=x_p\biggr\} \, .
  \]
  Three cases arise:
  \begin{enumerate}
  \item\label{it:hypersimplex:facet:0} $F_p\cap H$ is a facet of
   $F_p\cap\Hypersimplex{k}{n}\cong\Hypersimplex{k}{n-1}$ defined by $x_q=0$ (with
    $q\ne p$),
  \item\label{it:hypersimplex:facet:1} $F_p\cap H$ is a facet of
    $F_p\cap\Hypersimplex{k}{n}\cong\Hypersimplex{k}{n-1}$ as defined by
    Equation~\eqref{eq:hypersimplex:facet}, or
  \item\label{it:hypersimplex:split} $F_p\cap H$ defines a split of
    $F_p\cap\Hypersimplex{k}{n}\cong\Hypersimplex{k}{n-1}$.
  \end{enumerate}
  If $F_p\cap H$ is of type \eqref{it:hypersimplex:facet:0} then it follows that
  $\alpha_i=0$ for all $i\ne p$ and $\alpha_p\ne 0$.  As not all the $\alpha_i$ can vanish
  there is at most one $p\in[n]$ such that $F_p\cap H$ is of type
  \eqref{it:hypersimplex:facet:0}.  Since we could assume that $n\ge 4$ there are at least
  two distinct $p,q\in[n]$ such that $F_p\cap H$ and $F_q\cap H$ are of type
  \eqref{it:hypersimplex:facet:1} or \eqref{it:hypersimplex:split}.  By symmetry, we can
  further assume that $p=1$ and $q=n$.  So we get a partition $(A,B)$ of $[n-1]$ and a
  partition $(A',B')$ of $\{2,3,\dots,n\}$ with $\mu,\mu'\in\NN$ such that
  $F_1\cap H$ is defined by $x_1=0$ and
  \[
  \mu \sum_{i\in A} x_i \ = \ (k-\mu) \sum_{i\in B} x_i \, ,
  \]
  while $F_n\cap H$ is defined by $x_n=0$ and
  \[
  \mu' \sum_{i\in A'} x_i \ = \ (k-\mu') \sum_{i\in B'} x_i \, .
  \]
  We infer that there is a real number $\lambda$ such that $\alpha_i=\lambda\mu$ for all
  $i\in A$, $\alpha_i=\lambda(k-\mu)$ for all $i\in B$.  It remains to show that
  $\alpha_n\in\{\lambda\mu,\lambda(k-\mu)\}$.  Similarly, there is a real number
  $\lambda'$ such that $\alpha_i=\lambda'\mu'$ for all $i\in A'$,
  $\alpha_i=\lambda'(k-\mu')$ for all $i\in B'$.  As $n\ge 4$ we have $A\cap
  A'\ne\emptyset$ or $B\cap B'\ne\emptyset$.  We obtain $\alpha_i=\lambda\mu=\lambda'\mu'$
  for $i\in (A\cap A')\cup (B\cap B')$.  Finally, this shows that
  $\alpha_n\in\{\lambda'\mu',\lambda'(k-\mu')\}=\{\lambda\mu,\lambda(k-\mu)\}$, and this
  completes the proof.
\end{proof}

\begin{theorem}\label{thm:hypersimplex:n_splits}
  The total number of splits of the hypersimplex $\Hypersimplex{k}{n}$ (with $k\leq n/2$) equals
  \[
  \ (k-1)\left(2^{n-1}-(n+1)\right)-\sum_{i=2}^{k-1}(k-i)\binom{n}{i} \; .
  \]
\end{theorem}

\begin{proof}
  We have to count the $(A,B;\mu)$-hyperplanes with the restrictions listed in
  Lemma~\ref{lem:hypersimplex:splits}. So we take a set $A\subset[n]$ with at least 2 and
  at most $n-2$ elements. If $A$ has cardinality $i$ then there are
  $\min(k-1,n-i-1)-\max(1,k-i+1)+1$ choices for $\mu$. Recall that $(A,B;\mu)$ and
  $(B,A;k-\mu)$ define the same split; in this way we have counted each split twice. So we
  get
   \[ 
    \frac 12 \sum_{i=2}^{n-2}\big(\min(k,n-i)-\max(1,k-i+1)\big)\binom{n}{i}
    \ = \ \frac 12 \sum_{i=2}^{n-2}\big(\min(i,k,n-i)-1\big)\binom{n}{i}
  \]
  splits, where the equality holds since $k\leq n/2$. For a further simplification
  we rewrite the sum to get
  \begin{equation*}\begin{split}
    \frac12\sum_{i=2}^{k-1}&(i-1)\binom{n}{i}+\frac12\sum_{i=k}^{n-k}(k-1)\binom{n}{i}
    +\frac12\sum_{i=n-k+1}^{n-2}(n-i-1)\binom{n}{i}\\
    &=\ \frac 12(k-1)\sum_{i=2}^{n-2}\binom{n}{i} 
    +\frac12\sum_{i=2}^{k-1}\big(i-1-(k-1)\big)\binom{n}{i}+
    \frac12\sum_{i=n-k+1}^{n-2}\big(n-i-1-(k-1)\big)\binom{n}{i}\\
    &=\ (k-1)\left(2^{n-1}-(n+1)\right)-\sum_{i=2}^{k-1}(k-i)\binom{n}{i} \,.
  \end{split}\end{equation*}
\end{proof}

If we have two distinct splits $(A,B;\mu)$ and $(C,D;\nu)$ then either $\{A\cap C, A\cap D, B\cap
C, B\cap D\}$ is a partition of $[n]$ into four parts, or exactly one of the
four intersections is empty.  If, for instance, $B\cap D=\emptyset$ then $B\subseteq C$
and $D\subseteq A$.

\begin{proposition}\label{prop:hypersimplex:compatible}
  Two splits $(A,B;\mu)$ and $(C,D;\nu)$ of $\Hypersimplex{k}{n}$ are compatible if and
  only if one of the following holds:
    \begin{align*}
      \#(A\cap C) \ &\le \ k-\mu-\nu \, , &  \#(A\cap D) \ &\le \ \nu-\mu \, ,\\
      \#(B\cap C) \ &\le \ \mu-\nu   \, , &  \text{or }\quad \#(B\cap D) \ &\le \ \mu+\nu-k \, .
    \end{align*}
\end{proposition}

For an arbitrary set $I\subseteq[n]$ we abbreviate $x_I:=\sum_{i\in I}x_i$.  In
particular, $x_\emptyset=0$ and for $x\in\Hypersimplex{k}{n}$ one has $x_{[n]}=k$.

\begin{proof}
  Let $x\in\Hypersimplex{k}{n}$ be in the intersection of the $(A,B;\mu)$-hyperplane and
  the $(C,D;\nu)$-hyperplane.  Our split equations take the form
  \begin{align*}
    \mu(x_{A\cap C}+x_{A\cap D}) \ &= \ (k-\mu)(x_{B\cap C}+x_{B\cap D}) \quad \text{and}\\
    \nu(x_{A\cap C}+x_{B\cap C}) \ &= \ (k-\nu)(x_{A\cap D}+x_{B\cap D}) \, .
  \end{align*}
  In view of $(A\cap C)\cup(A\cap D)\cup(B\cap C)\cup(B\cap D)=[n]$ we additionally have
  $x_{A\cap C}+x_{A\cap D}+x_{B\cap C}+x_{B\cap D}=k$, and thus we arrive at the
  equivalent system of linear equations
  \begin{equation}\label{eq:hypersimplex:compatibility:0}
    x_{A\cap C} \ = \ k-\mu-\nu+x_{B\cap D} \, , \quad x_{A\cap D} \ = \ \nu-x_{B\cap D} \,
    , \quad \text{and} \quad x_{B\cap C} \ = \ \mu-x_{B\cap D}
  \end{equation}
  from which we can further derive
  \begin{equation}\label{eq:hypersimplex:compatibility:1}
    x_A \ = \ k-\mu \, , \quad x_B \ = \ \mu \, , \quad x_C \ = \ k-\nu \, , \quad \text{and}
    \quad x_D \ = \ \nu \, .
  \end{equation}
  Now the two given splits are \emph{incompatible} if and only if there exists a point
  $x\in(0,1)^n$ satisfying the conditions \eqref{eq:hypersimplex:compatibility:0}.

   Suppose first that none of the four intersections $A\cap C$,
  $A\cap D$, $B\cap C$, and $B\cap D$ is empty.  Then $x\in(0,1)^n$ satisfies the
  Equations~\eqref{eq:hypersimplex:compatibility:0} if and only if the system of
  inequalities in $x_{B\cap D}$
  \begin{align}\label{eq:hypersimplex:compatibility:2}
  \begin{array}{lclcl}
    0 & < & x_{B\cap D} & < &  \#(B\cap D)\\
    0 & < & k-\mu-\nu+x_{B\cap D} & < & \#(A\cap C) \\
    0 & < & \mu-x_{B\cap D} & < & \#(B\cap C) \\
    0 & < & \nu-x_{B\cap D} & < &\#(A\cap D)
  \end{array}
  \end{align}
  has a solution.  This is equivalent to the following system of inequalities:
  \[
  \begin{array}{lclcl}
    0 & < & x_{B\cap D} & < &  \#(B\cap D) \\
    \mu+\nu-k & < & x_{B\cap D} & < & \#(A\cap C)+\mu+\nu-k\\
    \mu-\#(B\cap C) & < & x_{B\cap D} & < & \mu \\
    \nu-\#(A\cap D) & < & x_{B\cap D} & < & \nu \, .
  \end{array}
  \]
  Obviously, the latter system admits a solution if and only if each of the four terms on
  the left is smaller than each of the four terms on the right.  Most of the resulting
  16~inequalities are redundant.  The following four inequalities remain
  \begin{align*}\label{eq:hypersimplex:compatibility:3}
    \#(A\cap C) \ &> \ k-\mu-\nu\\
    \#(A\cap D) \ &> \ \nu-\mu\\
    \#(B\cap C) \ &> \ \mu-\nu\\
    \#(B\cap D) \ &> \ \mu+\nu-k \, ,
  \end{align*}
  and this completes the proof of this case.
  
  For the remaining cases, we can assume by symmetry that $A\cap  C=\emptyset$.  
  Then $x\in(0,1)^n$ satisfies the
  Equations~\eqref{eq:hypersimplex:compatibility:0} if and only if $x_{B\cap
    D}=\mu+nu=-k$, $x_{A\cap D}=k-\mu$, and $x_{B\cap C}=k-\nu$. So the splits are not
  compatible if and only if
  \[
  \begin{array}{lclcl}
    0 & < & k-\mu & < &  \#(A\cap D)=\#A \\
    0 & < & k-\nu & < & \#(B\cap C)=\#C \\
    0 & < & \mu+\nu-k & < & \#(B\cap D)\,.
  \end{array}
  \]
  Since, by Lemma \ref{lem:hypersimplex:splits}, the first two inequalities hold for all
  splits this proves that the splits are compatible if and only if
  \[
  \#(A\cap C)=0\leq k-\mu-\nu \quad\text{or}\quad \#(B\cap D)\leq\mu+nu-k.
  \]
  However, again by using Lemma \ref{lem:hypersimplex:splits}, one has $\#(A\cap
  D)=\#A>k-\mu>\nu-\mu$, so $\#(A\cap D) \le\nu-\mu$ and, similarly, $\#(B\cap C)
  \le\mu-\nu$ cannot be true.  This completes the proof.
\end{proof}

In fact, the four cases of the proposition are equivalent in the sense that, by renaming
the four sets and exchanging $\mu$ and $\nu$ or $\mu$ and $k-\mu$ in a suitable way, one
will always be in the first case.

\begin{example}\label{ex:36:compatible}
  We consider the case $k=3$ and $n=6$.  For instance, the splits
  $(\{1,2,6\},\{3,4,5\};2)$ and $(\{4,5,6\},\{1,2,3\};2)$ are compatible since the
  intersection $\{3,4,5\}\cap\{1,2,3\}=\{3\}$ has only one element and $2+2-3=1$, that is,
  the inequality ``$\#(C\cap D)\le\mu+\nu-k$'' is satisfied.
\end{example}

\begin{corollary}
  Two splits $(A,B;\mu)$ and $(A,B;\nu)$ of $\Hypersimplex{k}{n}$ are always compatible.
\end{corollary}

\begin{proof}
  Without loss of generality we can assume that $\mu\ge\nu$.  Then the condition
  ``$\#(B\cap C) \le \mu-\nu$'' of Proposition~\ref{prop:hypersimplex:compatible} is
  satisfied.
\end{proof}

In Proposition \ref{prop:weak-split-complete} below we will show that the $1$-skeleton of
the weak split complex of any hypersimplex is always a complete graph.  In particular, the
weak split complex of $\Hypersimplex{k}{n}$ is connected. (Or it is void if
$k\in\{1,n-1\}$.)

\section{Finite Metric Spaces}\label{sec:metric}

This section revisits the classical case, studied in the papers by Bandelt and
Dress~\cite{BandeltDress86,BandeltDress92}; see also Isbell~\cite{Isbell64}.  Its purpose
is to show how some of the key results can be obtained as immediate corollaries to our
results above.

Let $\delta:\tbinom{[n]}{2}\to\nonnegRR$ be a metric on the finite set $[n]$; that is,
$\delta$ is a symmetric dissimilarity function which obeys the triangle inequality.  By
setting
\[
w_\delta(e_i+e_j) \ := \ -\delta(i,j)
\]
each metric $\delta$ defines a weight function $w_\delta$ on the second hypersimplex
$\Hypersimplex{2}{n}$.  Hence the results for $k=2$ from Section~\ref{sec:hypersimplices}
can be applied here.  The \emph{tight span} of $\delta$ is the tight span
$\tightspan{w_\delta}{\Hypersimplex{2}{n}}$.

Let $S=(A,B)$ be a \emph{split partition} of the set $[n]$, that is, $A,B\subseteq[n]$
with $A\cup B=[n]$, $A\cap B=\emptyset$, $\#A\ge2$, and $\#B\ge2$. This gives rise to the
\emph{split metric}
\[
\delta_S(i,j) \ := \
\begin{cases}
  0 & \text{if $\{i,j\}\subseteq A$ or $\{i,j\}\subseteq B$,} \\
  1 & \text{otherwise.}
\end{cases}
\]

The weight function $w_{\delta_S}=-\delta_S$ induces a split of the second hypersimplex
$\Hypersimplex{2}{n}$, which is induced by the $(A,B;1)$-hyperplane defined in
Equation~\eqref{eq:hypersimplex:split}.  Proposition~\ref{prop:hypersimplex:splits} now
implies the following characterization.

\begin{corollary}
  Each split of $\Hypersimplex{2}{n}$ is induced by a split metric.
\end{corollary}

Specializing the formula in Theorem~\ref{thm:hypersimplex:n_splits} with $k=2$ gives the
following.

\begin{corollary}
  The total number of splits of the hypersimplex $\Hypersimplex{2}{n}$ equals $2^{n-1}-n-1$.
\end{corollary}

The following corollary and proposition shows that our notions of compatibility and weak
compatibility agree with those of Bandelt and Dress~\cite{BandeltDress92} for in the
special case of~$\Hypersimplex{2}{n}$.

\begin{corollary}[Hirai~\cite{Hirai06}, Proposition 4.16]
  Two splits $(A,B)$ and $(C,D)$ of $\Hypersimplex{2}{n}$ are compatible if and only if
  one of the four sets $A\cap C$, $A\cap D$, $B\cap C$, and $B\cap D$ is empty.
\end{corollary}

\begin{proof}
  Let $(A,B)$ and $(C,D)$ be splits of $\Hypersimplex{2}{n}$.  We are in the situation of
  Proposition~\ref{prop:hypersimplex:compatible} with $k=2$ and $\mu=\nu=1$.  Hence all
  the right hand sides of the four inequalities in
  Proposition~\ref{prop:hypersimplex:compatible} yield zero, and this gives the claim.
\end{proof}

For a splits $S=(A,B)$ of $\Hypersimplex{2}{n}$ and $m\in [n]$ we denote by $S(m)$ that of 
the two set $A$, $B$ with $m\in S(m)$.

\begin{proposition}\label{prop:hypersimplex-weaklycompatibel}
  A set $\cS$ of splits of $\Hypersimplex{2}{n}$ is weakly compatible if
  and only if there does not exist $m_0,m_1,m_2,m_3\in[n]$ and $S_1,S_2,S_3\in\cS$ such
  that $m_i\in S_j$ if and only if $i=j$.
\end{proposition}
\begin{proof}
  This is the definition of a weakly compatible split system $\Hypersimplex{2}{n}$
  originally given by Bandelt and Dress in \cite[Section~1, page~52]{BandeltDress92}. Their Corollary 10 states that $\cS$ is weakly compatible in their sense if
  and only if $\sum_{S\in\cS} w_{S}$ is a coherent decomposition. However, this is our
  definition of weakly compatibility according to Lemma \ref{lem:weaklycompatible}.
\end{proof}

\begin{example}\label{ex:octahedron}
  The hypersimplex $\Hypersimplex{2}{4}$ is the regular octahedron, already studied in
  Example~\ref{ex:cross_d}.  It has the three splits $(\{1,2\},\{3,4\})$,
  $(\{1,3\},\{2,4\})$, and $(\{1,4\},\{2,3\})$.  The weak split complex is a triangle, and
  the split compatibility graph consists of three isolated points.
\end{example}

The split compatibility graph of $\Hypersimplex{2}{5}$ is isomorphic to the Petersen
graph.  It is shown in Figure~\ref{fig:petersen}.

\begin{figure}[htb]
  \includegraphics[width=.4\textwidth]{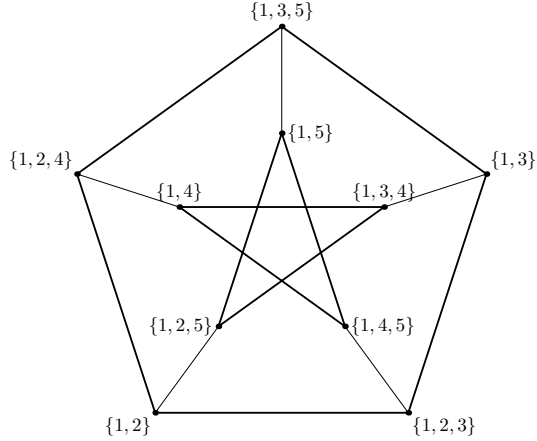}
  \caption{Split compatibility graph of $\Hypersimplex{2}{5}$; a split $(A,B)$ with $1\in
    A$ is labeled ``$A$''.}
  \label{fig:petersen}
\end{figure}

By Proposition \ref{prop:tree} each compatible system of splits gives rise to a tree. On
the other hand, given a tree with $n$ labeled leaves take for each edge $E$ that is not
connected to a leave the split $(A,B)$ where $A$ is the set of labels on one side of $E$
and $B$ the set of labels on the other side. So each tree gives rise to a system of splits for
$\Hypersimplex{2}{k}$ which is easily seen to be compatible.  This argument can be
augmented to a proof of the following theorem.

\begin{theorem}[Buneman~\cite{Buneman74}; Billera, Holmes, and Vogtmann~\cite{BilleraHolmesVogtmann01}]\label{thm:splitcomplex-trees}
  The split complex $\SplitComplex{\Hypersimplex{2}{n}}$ is the complex of trivalent
  leaf-labeled trees with $n$ leaves.
\end{theorem}

The split complex $\SplitComplex{\Hypersimplex{2}{n}}$ is equal to the \emph{link of the
  origin} $L_{n-1}$ of the \emph{space of phylogenetic trees} in
\cite{BilleraHolmesVogtmann01}.  It was proved in \cite[Theorem 2.4]{Vogtmann90} (see also
Robinson and Whitehouse~\cite{RobinsonWhitehouse96}) that
$\SplitComplex{\Hypersimplex{2}{n}}$ is homotopy equivalent to a wedge of $n-3$ spheres.
By a result of Trappmann and Ziegler, $\SplitComplex{\Hypersimplex{2}{n}}$ is even
shellable~\cite{TrappmannZiegler99}.  Markwig and Yu~\cite{MarkwigYu07} recently
identified the space of $k$ tropically collinear points in the tropical
$(d-1)$-dimensional affine space as a (shellable) subcomplex of
$\SplitComplex{\Hypersimplex{2}{k+d}}$.


\begin{example}
  Consider the split system $\cS=\smallSetOf{(A_{ij},[n]\setminus A_{ij})}{1\leq i < j
    \leq n\text{ and } j-i<n-2}$ where $A_{ij}:=\{i,i+1,\dots,j-1,j\}$ for the
  hypersimplex $\Hypersimplex{2}{n}$.  The combinatorial criterion of
  Proposition~\ref{prop:hypersimplex-weaklycompatibel} shows that this split system is
  weakly compatible, and that $\#\cS=\binom{n}{2}-n$. Since $\Hypersimplex{2}{n}$ has
  $\binom n2$ vertices and is of dimension $n-1$,
  Corollary~\ref{cor:weakly-compatible-triangulation} implies that $\subdivision{\cS}{P}$
  is a triangulation. This triangulation is known as the \emph{thrackle triangulation} in
  the literature; see~\cite{dLST95}, \cite[Chapter 14]{Sturmfels96}, and
  additionally~\cite{Stanley77,BandeltDress92,LamPostnikov05,HerrmannJoswig07} for further
  occurrences of this triangulation. In fact, as one can conclude from \cite[Theorem 3.1]{DHM00}
  in connection with \cite[Theorem 5]{BandeltDress92}, this is the only split
  triangulation of $\Hypersimplex{2}{n}$, up to symmetry.
\end{example}

\section{Matroid Polytopes and Tropical Grassmannians}\label{sec:matroid}

In the following, we copy some information from Speyer and
Sturmfels~\cite{SpeyerSturmfels04}; the reader is referred to this source for the details.

Let $\ZZ[p]:=\ZZ[p_{i_1,\dots,i_k}\,|\,1\le i_1<i_2<\dots<i_k\le n]$ be the polynomial
ring in $\tbinom{n}{k}$ indeterminates with integer coefficients.  The indeterminate
$p_{i_1,\dots,i_k}$ can be identified with the $k\times k$-minor of a $k\times n$-matrix
with columns numbered $(i_1,i_2,\dots,i_k)$.  The \emph{Pl\"ucker ideal} $I_{k,n}$ is
defined as the ideal generated by the algebraic relations among these minors.  It is
obviously homogeneous, and it is known to be a prime ideal.  For an algebraically closed
field $K$ the projective variety defined by $I_{k,n} \otimes_\ZZ K$ in the polynomial ring
$K[p]=\ZZ[p] \otimes_\ZZ K$ is the \emph{Grassmannian} $\Grassmannian{k}{n}$ (over $K$).
It parameterizes the $k$-dimensional linear subspaces of the vector space~$K^n$.

For instance, we can pick $K$ as the algebraic closure of the field $\CC(t)$ of rational
functions. Then for an arbitrary ideal $I$ in $K[x]=K[x_1,\dots,x_m]$ its
\emph{tropicalization} $\Tropicalization{I}$ is the set of all vectors $w\in\RR^m$ such
that the initial ideal $\initial_w(I)$ with respect to the term order defined by the
weight function~$w$ does not contain any monomial.  The \emph{tropical Grassmannian}
$\tropGrassmannian{k}{n}$ (over $K$) is the tropicalization of the Pl\"ucker ideal
$I_{k,n} \otimes_\ZZ K$.

The tropical Grassmannian $\tropGrassmannian{k}{n}$ is a polyhedral fan in
$\RR^{\tbinom{n}{k}}$ such that each of its maximal cones has dimension $(n-k)k+1$.  In a
way the fan $\tropGrassmannian{k}{n}$ contains redundant information. We describe the
three step reduction in \cite[Section~3]{SpeyerSturmfels04}.

Let $\phi$ be the linear map from $\RR^n$ to $\RR^{\tbinom{n}{k}}$ which sends
$x=(x_1,\dots,x_n)$ to $(x_I\,|\, I \in \tbinom{n}{k})$. Recall that $x_I$ is defined as
$\sum_{i\in I}x_i$.  The map $\phi$ is injective, and its image $\image\phi$ coincides with
the intersection of all maximal cones in $\tropGrassmannian{k}{n}$.  Moreover, the vector
$\vones:=(1,1,\dots,1)$ of length $\tbinom{n}{k}$ is contained in the image of $\phi$.
This leads to the definition of the two quotient fans
\[
\tropGrassmannianOne{k}{n} \ := \ \tropGrassmannian{k}{n}/\RR\vones \quad \text{and} \quad
\tropGrassmannianTwo{k}{n} \ := \ \tropGrassmannian{k}{n}/\image\phi \, .
\]
Finally, let $\tropGrassmannianThree{k}{n}$ be the (spherical) polytopal complex arising
from intersecting $\tropGrassmannianTwo{k}{n}$ with the unit sphere in
$\RR^{\tbinom{n}{k}}/\image\phi$.  We have $\dim\tropGrassmannianThree{k}{n}=n(k-1)-k^2$.
It seems to be common practice to use the name ``tropical Grassmannian'' interchangeably
for $\tropGrassmannian{k}{n}$, $\tropGrassmannianOne{k}{n}$, $\tropGrassmannianTwo{k}{n}$,
as well as $\tropGrassmannianThree{k}{n}$.

It is unlikely that it is possible to give a complete combinatorial description of all
tropical Grassmannians.  The contribution of combinatorics here is to provide kind of an
``approximation'' to the tropical Grassmannians via matroid theory.  For a background on
matroids, see the books edited by White~\cite{White86,White92}.

The \emph{tropical pre-Grassmannian} $\tropPreGrassmannian{k}{n}$ is the subfan of the
secondary fan of $\Hypersimplex{k}{n}$ of those weight functions which induce matroid
subdivisions.  A polytopal subdivision $\Sigma$ of $\Hypersimplex{k}{n}$ is a
\emph{matroid subdivision} if each (maximal) cell is a matroid polytope.  If $M$ is a
matroid on the set $[n]$ then the corresponding \emph{matroid polytope} is the convex hull
of those $0/1$-vectors in $\RR^n$ which are characteristic functions of the bases of~$M$.
A finite point set $X\subset\RR^d$ (possibly with multiple points) gives rise to a matroid
$\cM(X)$ by taking as bases for $\cM(X)$ the maximal affinely independent subsets of $X$.
The following characterization of matroid subdivisions is essential.

\begin{theorem}[Gel{\cprime}fand, Goresky, MacPherson, and Serganova~\cite{GGMS87},
  Theorem~4.1] \label{thm:matroidal} Let $\Sigma$ be a polytopal subdivision of
  $\Hypersimplex{k}{n}$.  The following are equivalent:
  \begin{enumerate}
  \item The maximal cells of $\Sigma$ are matroid polytopes, that is, $\Sigma $ is a
    matroid subdivision,
  \item the $1$-skeleton of $\Sigma$ coincides with the $1$-skeleton of
    $\Hypersimplex{k}{n}$, and
  \item the edges in $\Sigma$ are parallel to the edges of $\Hypersimplex{k}{n}$.
  \end{enumerate}
\end{theorem}

Regular matroid subdivisions of hypersimplices are called ``generalized Lie complexes'' by
Kapranov~\cite{Kapranov93}. The corresponding equivalence classes of weight functions are
the ``tropical Pl\"ucker vectors'' of Speyer~\cite{Speyer04}.

The relationship between the two fans $\tropPreGrassmannian{k}{n}$ and
$\tropGrassmannian{k}{n}$ is the following.  Algebraically, $\tropPreGrassmannian{k}{n}$
is the tropicalization of the ideal of quadratic Pl\"ucker relations; see
Speyer~\cite[Section~2]{Speyer04}.  Conversely, each weight function in the fan
$\tropGrassmannian{k}{n}$ gives rise to a matroid subdivision of~$\Hypersimplex{k}{n}$.
However, since there is no secondary fan naturally associated with
$\tropGrassmannian{k}{n}$ it is a priori not clear how $\tropGrassmannian{k}{n}$ sits
inside $\tropPreGrassmannian{k}{n}$. Note that, unlike $\tropGrassmannian{k}{n}$, the
tropical pre-Grassmannian does not depend on the characteristic of the field $K$.

Our goal for the rest of this section is to explain how the hypersimplex splits are
related to the tropical (pre-)Grassmannians.

\begin{proposition}\label{prop:split:refine}
  Let $\Sigma$ be a matroid subdivision and $S$ a split of $\Hypersimplex{k}{n}$. Then $\Sigma$ and $S$
  have a common refinement (without new vertices).
\end{proposition}

\begin{proof}
  Of course, one can form the common refinement $\Sigma'$ of $\Sigma$ and $S$ but
  $\Sigma'$ may contain additional vertices, and hence does not have to be a polytopal
  subdivision of $\Hypersimplex{k}{n}$. However, additional vertices can only occur if
  some edge of $\Sigma$ is cut by the hyperplane $H_S$. By Theorem~\ref{thm:matroidal}, all
  edges of $\Sigma$ are edges of $\Hypersimplex{k}{n}$. But since $S$ is a split, it does
  not cut any edges of $\Hypersimplex{k}{n}$.  Therefore $\Sigma'$ is a common refinement
  of $S$ and $\Sigma$ without new vertices.
\end{proof}

In order to continue, we recall some notions from linear algebra: Let $V$ be vector space.
A set $A\subset V$ is said to be in \emph{general position} if any subset $S$ of $B$ with
$\#S\leq \dim V+1$ is affinely independent. A family $\cA=\smallSetOf{A_i}{i\in I}$ in $V$
is said to be in \emph{relative general position} if for each affinely dependent set
$S\subseteq \bigcup_{i\in I} A_i$ with $\#S\leq \dim V+1$ there exists some $i\in I$ such
that $S\cap A_i$ is affinely dependent.

\begin{lemma}\label{lem:matroid-general-position}
  Let $\cM$ be a matroid of rank $k$ defined by $X\subset \RR^{k-1}$. If there exists some
  family $\cA=\smallSetOf{A_i}{i\in I}$ of sets in general position with respect to
  $X:=\bigcup_{i\in I} A_i$ such that each $A_i$ is in general position as a subset of
  $\aff A_i$ then the set of bases of $\cM$ is given by
  \begin{align}\label{eq:matroid-general-position}
    \SetOf{B\subset X}{\#B=k\text{ and } \#(B\cap A_i)\leq \dim \aff A_i+1\text{ for all }i\in I} \, .
  \end{align}
\end{lemma}

\begin{proof}
  It is obvious that for each basis $B$ of $\cM$ one has $\#(B\cap A_i)\leq \dim \aff
  A_i+1$ for all $i\in I$. So it remains to show that each set $B$ in
  \eqref{eq:matroid-general-position} is affinely independent. Let $B$ be such a set and
  suppose that $B$ is not affinely independent. Since $\cA$ is in relative general
  position there exists some $i\in I$ such that $B\cap A_i$ is affinely dependent.
  However, since $\#(B\cap A_i)\leq \dim \aff A_i+1$, this contradicts the fact that $A_i$
  is in general position in $\aff A_i$.
\end{proof}

From each split $(A,B;\mu)$ of $\Hypersimplex{k}{n}$ we construct two matroid polytopes
with points labeled by $[n]$: Take any $(\mu-1)$-dimensional (affine) subspace $U\subset
\RR^{k-1}$ and put $\#B$ points labeled by $B$ into $U$ such that they are in general
position (as a subset of $U$). The remaining points, labeled by $A$, are placed in
$\RR^{k-1}\setminus U$ such that they are in general position and in relative general
position with respect to the set of points labeled by $B$. By Lemma
\ref{lem:matroid-general-position} the bases of the corresponding matroid are all
$k$-element subsets of $[n]$ with at most $\mu$ points in $B$. These are exactly the
points in one side of \eqref{eq:hypersimplex:split}.  The second matroid is obtained
symmetrically, that is, starting with $\#A$ points in a $(k-\mu-1)$-dimensional subspace.
Since splits are regular and correspond to rays in the secondary fan we have proved the
following lemma.

\begin{lemma}\label{lem:splits-matroids}
  Each split of $\Hypersimplex{k}{n}$ defines a regular matroid subdivision and hence
  a ray in $\tropPreGrassmannian{k}{n}$.
\end{lemma}

Matroids arising in this way are called \emph{split matroids}, and the corresponding
matroid polytopes are the \emph{split matroid polytopes}.

\begin{remark}
  Kim \cite{Kim08} studies the splits of general matroid polytopes. However, his
  definition of a split requires that it induces a matroid subdivision.
  Lemma~\ref{lem:splits-matroids} shows that for the entire hypersimplex these notions
  agree.  In this case, \cite[Theorem 4.1]{Kim08} reduces to our
  Lemma~\ref{lem:hypersimplex:splits}.
\end{remark}

\begin{proposition}\label{prop:weak-split-complete}
  The $1$-skeleton of the weak split complex $\WeakSplitComplex{\Hypersimplex{k}{n}}$ of $\Hypersimplex{k}{n}$
  is a complete graph.
\end{proposition}

\begin{proof}
  We have to prove that any two splits of $\Hypersimplex{k}{n}$ are weakly compatible.
  Since splits are matroid subdivisions by Lemma~\ref{lem:splits-matroids} this
  immediately follows from Proposition~\ref{prop:split:refine}.
\end{proof}

\begin{example}\label{ex:octahedron:subdivision}
  We continue our Example~\ref{ex:octahedron}, where $k=2$ and $n=4$.  Up to symmetry,
  each split of the regular octahedron $\Hypersimplex{2}{4}$ looks like
  $(\{1,2\},\{3,4\};1)$, that is, $\mu=1$.

  In this case, the affine subspace $U$ is just a single point on the line $\RR^1$.  The
  only choice for the two points corresponding to $B=\{3,4\}$ is the point $U$ itself.
  The two points corresponding to $A=\{1,2\}$ are two arbitrary distinct points both of
  which are distinct from $U$.  The situation is displayed in
  Figure~\ref{fig:octahedron:subdivision} on the left.  This defines the first of the two
  matroids induced by the split $(\{1,2\},\{3,4\};1)$.  Its bases are $\{1,2\}$,
  $\{1,3\}$, $\{1,4\}$, $\{2,3\}$, and $\{2,4\}$.

  The second matroid is obtained in a similar way.  Both matroid polytopes are square
  pyramids, and they are shown (with their vertices labeled) in
  Figure~\ref{fig:octahedron:subdivision} on the right.  The pyramid in bold is the one
  corresponding to the matroid whose construction has been explained in detail above and
  which is shown on the left.
\end{example}

\begin{figure}[htb]
  \strut
  \hfill
  \begin{minipage}[c]{.35\textwidth}
    \includegraphics[width=\textwidth]{matroid.0}
  \end{minipage}
  \hfill
  \begin{minipage}[c]{.45\textwidth}
    \includegraphics[width=\textwidth]{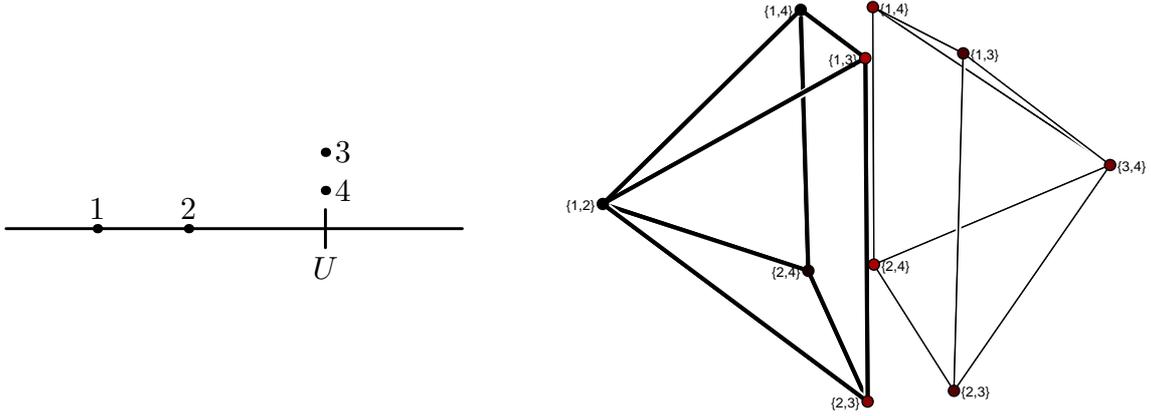}
  \end{minipage}
  \hfill\strut
  \caption{Matroid and matroid subdivision induced by a split as explained in
    Example~\ref{ex:octahedron:subdivision}.}
  \label{fig:octahedron:subdivision}
\end{figure}

As in the case of the tropical Grassmannian, we can intersect the fan
$\tropPreGrassmannian{k}{n}$ with the unit sphere in $\RR^{\tbinom{n}{k}-n}$ to arrive at a
(spherical) polytopal complex $\tropPreGrassmannianOne{k}{n}$, which we also call the
\emph{tropical pre-Grassmannian}.  The following is one of our main results.

\begin{theorem}\label{thm:splitsubcomplex}
  The split complex $\SplitComplex{\Hypersimplex{k}{n}}$ is a polytopal subcomplex of the
  tropical pre-Grassmannian $\tropPreGrassmannianOne{k}{n}$.
\end{theorem}

\begin{proof}
  By Proposition~\ref{cor:split-secondary-complex}, the split complex is a subcomplex of
  $\SecondaryFanOne{\Hypersimplex{k}{n}}$. Furthermore, by Lemma~\ref{lem:splits-matroids}
  each split corresponds to a ray of $\tropPreGrassmannian{k}{n}$. So it remains to show
  that all maximal cells of $\subdivision{\cS}{\Hypersimplex{k}{n}}$ are matroid polytopes
  whenever $\cS$ is a compatible system of splits. The proof will proceed by induction on
  $k$ and~$n$. Note that, since $\Hypersimplex{k}{n}\cong\Hypersimplex{n-k}{n}$, it is
  enough to have as base case $k=2$ and arbitrary $n$, which is given by
  Proposition~\ref{prop:tropGrassmannian:2:n}.
  
  By Theorem \ref{thm:matroidal}, we have to show that there do not occur any edges in
  $\subdivision{\cS}{\Hypersimplex{k}{n}}$ that are not edges of $\Hypersimplex{k}{n}$.
  Since $\cS$ is compatible no split hyperplanes meet in the interior of
  $\Hypersimplex{k}{n}$, and so additional edges could only occur in the boundary. By
  Observation~\ref{obs:facets}, for each split $S\in\cS$ and each facet $F$ of
  $\Hypersimplex{k}{n}$ there are two possibilities: Either $H_S$ does not meet the
  interior of~$F$, or $H_S$ induces a split $S'$ on $F$. The restriction of
  $\subdivision{\cS}{\Hypersimplex{k}{n}}$ to $F$ equals the common refinement of all such
  splits $S'$. So, using the induction hypothesis and again Theorem~\ref{thm:matroidal},
  it suffices to prove that the split systems that arise in this fashion are compatible.
  
  So let $S=(A,B,\mu)\in\cS$. We have to consider to types of facets of
  $\Hypersimplex{k}{n}$ induced by $x_i=0$, $x_i=1$, respectively. In the first case, the
  arising facet $F$ is isomorphic to $\Hypersimplex{k}{n-1}$ and, if $H_S$ meets $F$ in
  the interior, the split $S'$ of $F$ equals $(A\setminus \{i\}, B;\mu)$ or $(A,
  B\setminus \{i\};\mu)$. It is now obvious by
  Proposition~\ref{prop:hypersimplex:compatible} that the system of all such $S'$ is
  compatible if $\cS$ was.
  
  In the second case, the facet $F$ is isomorphic to $\Hypersimplex{k-1}{n-1}$ and $S'$
  (again if $H_S$ meets the interior of $F$ at all) equals $(A\setminus \{i\}, B;\mu)$ or
  $(A, B\setminus \{i\};\mu-1)$. To show that a split system is compatible it suffices to
  show that any two of its splits are compatible. So let $S=(A,B;\mu)$ and $T=(C,D;\nu)$
  be compatible splits for $\Hypersimplex{k}{n}$ such that $H_S$ and $H_T$ meet the
  interior of $F$, and $S'=(A',B';\mu')$, $T'=(C',D';\nu')$, respectively, the
  corresponding splits of $F$. By the remark after
  Proposition~\ref{prop:hypersimplex:compatible}, we can suppose that we are in the first
  case of Proposition~\ref{prop:hypersimplex:compatible}, that is, $\#(A\cap C)\le
  k-\mu-\nu$. We now have to consider the four cases that $i$ is an element of either
  $A\cap C$, $A\cap D$, $B\cap C$, or $B \cap D$. In the first case, we have
  $S'=(A\setminus \{i\}, B;\mu)$ and $T'=(C\setminus \{i\},D,\nu)$. We get $\# (A'\cap C')
  = \#(A\cap B)-1\leq k-\mu-\nu-1= (k-1)-\mu'-\nu'$, so $S'$ and $T'$ are compatible. The
  other cases follow similarly, and this completes the proof of the theorem.
\end{proof}  

\begin{construction}\label{con:splitsubcomplex}
  We will now explicitly construct the matroid polytopes that occur in the refinement of
  two compatible splits. So consider two compatible splits of $\Hypersimplex{k}{n}$
  defined by an $(A,B;\mu)$- and a $(C,D;\nu)$-hyperplane. These two hyperplanes divide
  the space into four (closed) regions. Compatibility implies that the intersection of one
  of these regions with $\Hypersimplex{k}{n}$ is not full-dimensional, two of the
  intersections are split matroid polytopes, and the last one is a full-dimensional
  polytope of which we have to show that it is a matroid polytope.  It therefore suffices
  to show that one of the four intersections is a full-dimensional matroid polytope that
  is not a split matroid polytope.
  
  By Proposition \ref{prop:hypersimplex:compatible} and the remark following its proof, we
  can assume without loss of generality that $\#(B\cap D) \le \mu+\nu-k$.  Note first
  that the equation $\sum_{i\in B}x_i=\mu$ also defines the $(A,B;\mu)$-hyperplane
  from Equation~\eqref{eq:hypersimplex:split}, since $x_{A\cup B}=k$ for any point
  $x\in\Hypersimplex{k}{n}$. We will show that the intersection of $\Hypersimplex{k}{n}$
  with the two halfspaces defined by
  \[
  \sum_{i\in B}x_i \ \leq \ \mu \qquad\text{and}\qquad \sum_{i\in D}x_i \ \leq \ \nu
  \]
  is a full dimensional matroid polytope which is not a split matroid polytope.

  To this end, we define a matroid on the ground set $[n]$ together with a realization
  in~$\RR^{k-1}$ as follows. Pick a pair of (affine) subspaces $U_B$ and $U_D$ of
  $\RR^{k-1}$ such that the following holds: $\dim U_B=\mu-1$, $\dim U_D=\nu-1$, and $\dim
  (U_B\cap U_D)=\mu+\nu-k-1$. Note that the latter expression is non-negative as
  $0\leq\#(B\cap D)\leq \mu+\nu-k-1$. The dimension formula then implies that $\dim
  (U_B+U_D)=\mu-1+\nu-1-\mu-\nu+k+1=k-1$, that is, $U_B+U_D=\RR^{k-1}$.

  Each element in $[n]$ labels a point in $\RR^{k-1}$ according to the following
  restrictions.  For each element in the intersection $B\cap D$ we pick a point in
  $U_B\cap U_D$ such that the points with labels in $B\cap D$ are in general position
  within $U_B\cap U_D$.  Since $\#(B\cap D) \le \mu+\nu-k$ the points with labels in
  $B\cap D$ are also in general position within $U_B$.  Therefore, for each element in
  $B\setminus D=B\cap C$ we can pick a point in $U_B\setminus(U_B\cap U_D)$ such that all
  the points with labels in $B$ are in general position within $U_B$.  Similarly, we can
  pick points for the elements of $D\cap A$ in $U_D\setminus (U_B\cap U_D)$ such that the
  points with labels in~$D$ are in general position within~$U_D$.  Without loss of
  generality, we can assume that the points with labels in $B$ and the points with labels
  in~$D$ are in relative general position as subsets of $U_B+U_D=\RR^{k-1}$.

  For the remaining elements in $A\cap C=[n]\setminus (B\cup D)$ we can pick points in
  $\RR^{k-1} \setminus (U_B\cup U_D)$ such that the points with labels in $A\cap C$ are in
  general position and the family of sets of points with labels in $B$, $D$, and $A\cap
  C$, respectively, is in relative general position.  By Lemma
  \ref{lem:matroid-general-position} the matroid generated by this point set has the
  desired property.
\end{construction}

\begin{example}\label{ex:36:matroid}
  We continue our Example~\ref{ex:36:compatible}, where $k=3$ and $n=6$, considering the
  compatible splits $(\{1,2,6\},\{3,4,5\};2)$ and $(\{4,5,6\},\{1,2,3\};2)$.  In the
  notation used in Construction~\ref{con:splitsubcomplex} we have $A=\{1,2,6\}$, $B=\{3,4,5\}$,
  $C=\{4,5,6\}$, $D=\{1,2,3\}$, and $\mu=\nu=2$.  Hence $A\cap C=\{6\}$, $A\cap
  D=\{1,2\}$, $B\cap C=\{4,5\}$, and $B\cap D=\{3\}$. The matroid from
  Construction~\ref{con:splitsubcomplex} is displayed in Figure~\ref{fig:36:matroid}.  The
  non-split matroid polytope constructed in the proof of Theorem~\ref{thm:splitsubcomplex}
  has the $f$-vector $(18,72,101,59,14)$.
\end{example}

\begin{figure}[htb]
  \includegraphics[width=.35\textwidth]{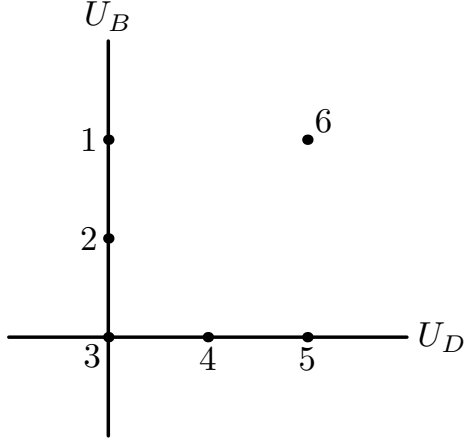}

  \caption{Non-split matroid constructed from two compatible splits in
    $\Hypersimplex{3}{6}$ as in Example~\ref{ex:36:matroid}.}
  \label{fig:36:matroid}
\end{figure}

For the special case $k=2$ the structure of the tropical Grassmannian and pre-Grassmannian
is much simpler. The following proposition follows from \cite[Theorem
3.4]{SpeyerSturmfels04}, in connection with Theorem~\ref{thm:splitcomplex-trees}.

\begin{proposition}\label{prop:tropGrassmannian:2:n}
  The tropical Grassmannian $\tropGrassmannianThree{2}{n}$ equals
  $\tropPreGrassmannianOne{2}{n}$, and it is a simplicial complex which is isomorphic to
  the split complex $\SplitComplex{\Hypersimplex{2}{n}}$.
\end{proposition}

Let us revisit the two smallest cases: The tropical Grassmannian
$\tropGrassmannianThree{2}{4}$ consists of three isolated points corresponding to the
three splits of the regular octahedron, and $\tropGrassmannianThree{2}{5}$ is a
$1$-dimensional simplicial complex isomorphic to the Petersen graph; see
Figure~\ref{fig:petersen}.

\begin{proposition}
  The rays in $\tropPreGrassmannian{k}{n}$ correspond to the coarsest regular matroid
  subdivisions of $\Hypersimplex{k}{n}$.
\end{proposition}

\begin{proof}
  By definition, a ray in $\tropPreGrassmannian{k}{n}$ defines a regular matroid
  subdivision which is coarsest among the matroid subdivisions of  $\Hypersimplex{k}{n}$.
  We have to show that this is a coarsest among all subdivisions.

  To the contrary, suppose that $\Sigma$ is a coarsest matroid subdivision which can be
  coarsened to a subdivision $\Sigma'$.  By construction the $1$-skeleton of $\Sigma'$ is
  contained in the $1$-skeleton of $\Sigma$.  From Theorem~\ref{thm:matroidal} it follows
  that $\Sigma'$ is matroidal.  This is a contradiction to $\Sigma$ being a coarsest
  matroid subdivision.
\end{proof}

\begin{example}\label{ex:G36}
  In view of Proposition~\ref{prop:tropGrassmannian:2:n}, the first example of a tropical
  Grassmannian that is not covered by the previous results is the case $k=3$ and $n=6$.
  So we want to describe how the split complex $\SplitComplex{\Hypersimplex{3}{6}}$ is
  embedded into $\tropGrassmannianThree{3}{6}$.  We use the notation
  of~\cite[Section~5]{SpeyerSturmfels04}; see also~\cite[Section~4.3]{Speyer05}.

  The tropical Grassmannian $\tropGrassmannianThree{3}{6}$ is a pure $3$-dimensional
  simplicial complex which is not a flag complex.  Its $f$-vector reads
  $(65,550,1395,1035)$, and its homology is concentrated in the top dimension.  The only
  non-trivial (reduced) homology group (with integral coefficients) is
  $H_3(\tropGrassmannianThree{3}{6};\ZZ)=\ZZ^{126}$.

  The splits with $A=\{1\}\cup A_1$, $\mu=1$, and $A=\{1\}\cup A_3$, $\mu=2$, are the 15
  vertices of type ``F''.  The splits with $A=\{1\}\cup A_2$ and $\mu\in\{1,2\}$ are the
  20 vertices of type ``E''.  Here $A_m$ is an $m$-element subset of $\{2,3,\dots,n\}$.
  The remaining 30 vertices are of type ``G'', and they correspond to coarsest
  subdivisions of $\Hypersimplex{3}{6}$ into three maximal cells.  Hence they do not occur
  in the split complex.  See also Billera, Jia, and Reiner~\cite[Example
  7.13]{BilleraJiaReiner06}.

  The 100 edges of type ``EE'' and the 120 edges of type ``EF'' are the ones induced by
  compatibility.  Since $\SplitComplex{\Hypersimplex{3}{6}}$ does not contain any
  ``FF''-edges it is not an induced subcomplex of $\tropGrassmannianThree{3}{6}$.  The
  matroid shown in Figure~\ref{fig:36:matroid} arises from an ``EE''-edge.

  The split complex is $3$-dimensional and not pure; it has the $f$-vector
  $(35,220,360,30)$.  The 30 facets of dimension $3$ are the tetrahedra of type
  ``EEEE''.  The remaining 240 facets are ``EEF''-triangles.

  The integral homology of $\SplitComplex{\Hypersimplex{3}{6}}$ is concentrated in
  dimension two, and it is free of degree $144$.
\end{example}

\begin{remark}
  Example \ref{ex:G36} and Proposition \ref{prop:tropGrassmannian:2:n} show that the split
  complex is a subcomplex of $\tropGrassmannianThree{k}{n}$ if $d=2$ or $n\leq 6$.
  However, this does not hold in general: Consider the weight functions $w,w'$ defined in
  the proof of \cite[Theorem 7.1]{SpeyerSturmfels04}.  It is easily seen from
  Proposition~\ref{prop:hypersimplex:compatible} that $w$ and $w'$ are the sum of the
  weight functions of compatible systems of vertex splits for $\Hypersimplex{3}{7}$.  Yet
  in the proof of \cite[Theorem 7.1]{SpeyerSturmfels04}, it is shown that
  $w,w'\not\in\tropGrassmannianThree{3}{7}$ for fields with characteristic not equal to
  $2$ and equal to $2$, respectively.
\end{remark}

\section{Open Questions and Concluding Remarks}

We showed that special split complexes of polytopes (e.g., of the polygons and of the
second hypersimplices) already occurred in the literature albeit not under this name. So
the following is natural to ask.

\begin{question}
  What other known simplicial complexes arise as split complexes of polytopes?
\end{question}

The split hyperplanes of a polytope define an affine hyperplane arrangement.  For example,
the coordinate hyperplane arrangements arises as the split hyperplane arrangement of the
cross polytopes; see Example~\ref{ex:cross_d}.

\begin{question}
  Which hyperplane arrangements arise as split hyperplane arrangements of some polytope?
\end{question}

Jonsson \cite{Jonsson05} studies generalized triangulations of polygons; this has a
natural generalization to simplicial complexes of split systems such that no $k+1$ splits
in such a system are totally incompatible. See also \cite{PilaudSantos07,Dress+07}.

\begin{question}
  How do such \emph{incompatibility complexes} look alike for other polytopes?
\end{question}

All computations with polytopes, matroids, and simplicial complexes were done with
\texttt{polymake}~\cite{polymake}.  The visualization also used
\texttt{JavaView}~\cite{javaview}.

We are indebted to Bernd Sturmfels for fruitful discussions. We also thank Hiroshi Hirai
and an anonymous referee for several useful comments.

\bibliographystyle{amsplain}
\bibliography{main}

\end{document}